%% file: main.tex
\documentclass[runningheads]{llncs}
\usepackage{graphicx}
\usepackage{xcolor}

\definecolor{webred}{rgb}{0.5,0,0}
\definecolor{webblue}{rgb}{0,0,0.8}
\usepackage[pdftex,colorlinks,citecolor=webblue,linkcolor=webred,]{hyperref}

\usepackage{color, colortbl}
\definecolor{Red}{rgb}{1.0,0.8,0.8}

\usepackage{amsmath}
\usepackage{amsfonts}
\usepackage{amssymb}
\usepackage{stmaryrd}
\usepackage[boxed,linesnumbered,lined,commentsnumbered]{algorithm2e}

\usepackage{afterpage}
\usepackage{longtable}
\usepackage{booktabs}
\usepackage{mathtools}

\usepackage{subfig}

\newif\ifarxiv

\arxivtrue % or
\begin{document}
%
%\title{Contribution Title\thanks{Supported by organization x.}}
%\title{Verifying Closed-Loop ACAS Xu NNCS with State Quantization and Backreachability}
\title{Neural Network Compression of ACAS Xu Early Prototype is Unsafe: Closed-Loop Verification through 
Quantized State Backreachability}
\titlerunning{Neural Network Compression of ACAS Xu Early Prototype is Unsafe}
%\author{First Author\inst{1}\orcidID{0000-1111-2222-3333} \and
%Second Author\inst{2,3}\orcidID{1111-2222-3333-4444} \and
%Third Author\inst{3}\orcidID{2222--3333-4444-5555}}
\author{Stanley Bak\inst{1} \and Hoang-Dung Tran\inst{2}}
\authorrunning{S. Bak and H.D. Tran}
\institute{Stony Brook University, Stony Brook, NY, USA \\
\email{stanley.bak@stonybrook.edu}\\
\and
University of Nebraska-Lincoln, Lincoln, NE, USA \\
\email{dtran30@unl.edu}}
\maketitle              % typeset the header of the contribution
\begin{abstract}
\input{abstract}
\end{abstract}
\input{intro}

\input{background}

\input{algorithm}

\input{results}

\input{related}
\input{conclusion}

\section*{Acknowledgement}
This material is based upon work supported by the NSF EPSCoR First Award, the Air Force Office of Scientific Research and the Office of Naval Research under award numbers FA9550-19-1-0288, FA9550-21-1-0121, FA9550-22-1-0450 and N00014-22-1-2156. Any opinions, findings, and conclusions or recommendations expressed in this material are those of the author(s) and do not necessarily reflect the views of the NSF, US Air Force or US Navy.

\bibliographystyle{splncs04}
\bibliography{main}

\ifarxiv
\appendix
\include{apx_counterexamples}
\fi

\end{document}

%% file: abstract.tex
ACAS Xu is an air-to-air collision avoidance system designed for unmanned aircraft that issues horizontal turn advisories to avoid an intruder aircraft.
Due the use of a large lookup table in the design, a neural network compression of the policy was proposed. 
Analysis of this system has spurred a significant body of research in the formal methods community on neural network verification.
While many powerful methods have been developed, most work focuses on open-loop properties of the networks, rather than the main point of the system---collision avoidance---which requires closed-loop analysis.

\vspace{1em}

In this work, we develop a technique to verify a closed-loop approximation of the system using \emph{state quantization} and \emph{backreachability}.
We use favorable assumptions for the analysis---perfect sensor information, instant following of advisories, ideal aircraft maneuvers and an intruder that only flies straight.
When the method fails to prove the system is safe, we refine the quantization parameters until generating counterexamples where the original (non-quantized) system also has collisions.
%
%As far as we are aware, this is the first work to find unsafe cases in the original ACAS Xu neural networks when used in closed-loop without noise.
\keywords{Neural Network Verification \and ACAS Xu \and Reachability}

%% file: intro.tex
\section{Introduction}

The Airborne Collision Avoidance System X (ACAS X) is a mid-air collision avoidance system under development~\cite{olson2015airborne}, with the ACAS Xu variant focused on collision avoidance for unmanned aircraft~\cite{katz2017reluplex}.
Originally designed offline using dynamic programming and Markov decision processes (MDPs)~\cite{kochenderfer2011robust}, the large rule table was compressed by a factor of 1000 using a set of neural networks~\cite{julian2016policy}.
The proposed system is an example of a neural network control system (NNCS), where the system's execution alternates between the aircraft dynamics and a neural network controller.
As collision avoidance is safety-critical, analysis of the neural networks has spurred a significant body of research on neural network verification.
%, with the original work gathering over 1000 citations in less than four years~\cite{katz2017reluplex}.
%
Most existing work, however, focuses on \emph{open-loop} verification, such as property $\phi_3$ from the original work~\cite{katz2017reluplex}, which states, ``if the intruder is directly ahead and is moving towards the ownship, [a turn will be commanded].''
Open-loop properties can be expressed in terms of constraints over the inputs and outputs of a single execution of the neural network.
However, satisfying open-loop properties does not prove the system is safe, as this requires reasoning with the physical system dynamics---how the aircraft responds to turn commands.
Also, the system is running continuously and may change advisories at a future time, complicating safety analysis.
%
%Existing work on closed-loop verification of NNCS often either assumes a small set of initial conditions or uses a simpler neural network compared with ACAS Xu~\cite{johnson2021arch}.
%
Verification of closed-loop safety of provided collision avoidance system under all designed operating conditions is thus a sort of grand challenge.
%
%In this paper, we try to answer: \emph{under ideal assumptions, does the closed-loop ACAS Xu NNCS system actually avoid collisions?}

While verification of neural networks is continuously improving, an intriguing alternate approach has recently been proposed based on input quantization~\cite{jia2021verifying}.
Rather than verifying the neural network directly, which requires reasoning about the semantics at each layer, the system's execution semantics are changed to round the inputs to a discrete set of possible values before running the network.
%, based on a fixed quantum value for each network input.
%
To be clear, this type of quantization is a preprocessing layer before the network runs; it does not change the representation of the floating-point values inside the network itself.
Through input quantization, proving open-loop properties of a neural network is reduced to the problem of \emph{network execution} for each of a finite set of possible inputs.
Due to the possibility of combinatorial explosion, this strategy can only work if the number of inputs is small, which is often the case for neural networks used in control systems.
When the strategy is applicable, however, it enjoys several advantages: 
(i) batch execution of neural networks is often used in training and so optimized hardware like GPUs can be leveraged to enumerate the possible inputs for verification, 
(ii) the performance of the final quantized system approximates the performance of the original neural network and the approximation can be tuned through the quantization parameters, and %
(iii) the verification method only requires execution, and works regardless of the network size, the network architecture, or the layer types, unlike most neural network verification methods.
In the context of verification, however, quantization has only been considered for open-loop properties.

In this work, we propose an approach to formally verify quantized closed-loop NNCS.
Although the technique is general, we focus primarily on proving safety for quantized version of the well-studied aircraft collision avoidance neural network benchmark.
Two key ideas are needed to make this work: (1) we perform \emph{state quantization} rather than input quantization and (2) we use \emph{backreachability} from the unsafe states to reduce the number of partitions.
We prove the approach is sound and complete, in the sense that by continuing to refine quantization parameters, either the quantized system will eventually be proven safe or an unsafe counterexample will be found in the original system.
When the method fails to prove safety of quantized closed-loop system, we refine the quantization values until discovering cases where the original (unquantized) version of the system fails.
We also show that with stricter assumptions on the ownship aircraft's velocity, the quantized system can guarantee safety.

%% file: background.tex
\section{Background and Problem Formulation}
We next review key aspects of the system design, proof assumptions, and provide background on $\mathcal{AH}$-Polytopes before formulating the safety verification problem.

\subsection{Collision Avoidance System Design} We are interested in safety verification and falsification of the \emph{closed-loop} air-to-air collision avoidance system~\cite{kochenderfer2011robust,katz2017reluplex} depicted in Figure~\ref{fig:closed-loop-ACASXu}. 
The system computes advisory commands to control an ownship aircraft with physical dynamics described by a set of ordinary differential equations (ODEs), trying to avoid collisions with a nearby intruder. 

\begin{figure}[t]
    \centering
    \includegraphics[width=0.99\columnwidth]{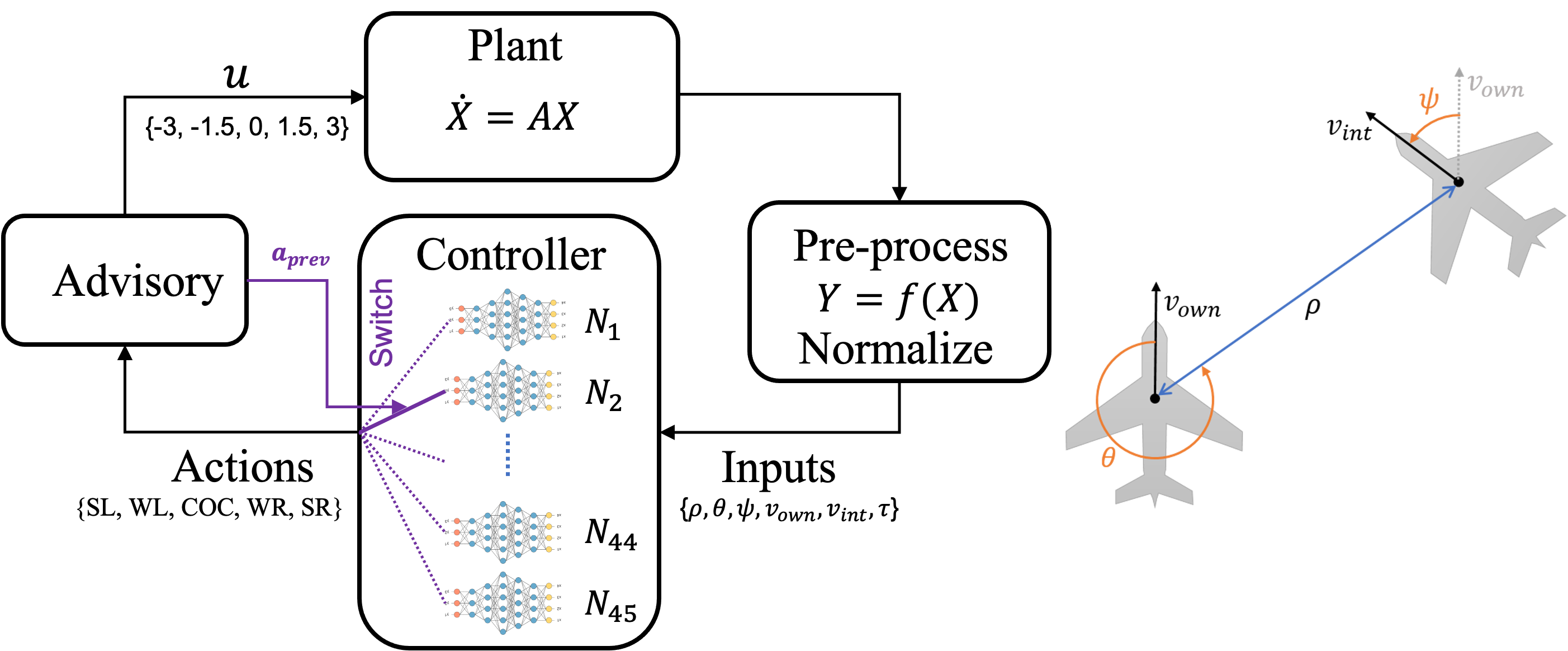}
    \caption{The closed-loop air-to-air collision avoidance system design.}
    \label{fig:closed-loop-ACASXu}
    %\vspace{-1.5em}
\end{figure}

A detailed description of the inputs and actions in the system is shown in Table \ref{tab:acasxu-inputs-actions}.
The system receives 7 inputs about the state of an ownship and a nearby intruder aircraft, $\mathcal{I} = \{\rho, \theta, \psi, v_{own}, v_{int}, \tau, a_{prev}\}$, and produces one of five possible advisories for the ownship, $\mathcal{A} = \{\textsc{coc, wl, wr, sl, sr}\}$.

The turn advisories in the system are generated by 45 deep ReLU neural networks with 6 layers and 50 neurons per layer for each network. 
Control switches between different neural networks $N_{a_{prev},\tau}$ based on the previous advisory $a_{prev}$ (total of 5 choices) and the time until loss of vertical separation $\tau = \{0, 1, 5, 10, 20, 50, 60, 80, 100\}$ (total of 9 choices). 
For example, the network $N_{5,3}$ will be invoked if the previous advisory is $a_{prev} = \textsc{sr}$ and $\tau = 5$. 
If the ownship and the intruder are at the same altitude, then $\tau = 0$ and only five neural network controllers need to be used, $N_{1,1}, N_{2,1}, N_{3,1}, N_{4,1}$, and $N_{5,1}$. 

\begin{table}[b]
\centering
  \resizebox{\linewidth}{!}{
\begin{tabular}{lll|ll}
\hline
\multicolumn{1}{c}{\bf Input} & \multicolumn{1}{c}{\bf Units} & \multicolumn{1}{c}{\bf Description} & \multicolumn{1}{|c}{\bf Action} & \multicolumn{1}{c}{\bf Description} \\ \hline
$\rho$ & $ft$  & distance between ownship and intruder                                &  \textsc{sl} & strong left turn at 3.0 deg/s                               \\
$\theta$ & $rad$ & angle to intruder w.r.t ownship heading ~~                               &  \textsc{wl} & weak left at turn 1.5 deg/s                                \\
$\psi$ & $rad$  & heading of intruder w.r.t ownship                            & \textsc{coc} &    clear of conflict (do nothing)                            \\
$v_{own}$ & $ft/s$ & velocity of ownship                             &  \textsc{wr}  & weak right turn at 1.5 deg/s                               \\
$v_{int}$ & $ft/s$ & velocity of intruder                             &  \textsc{sr}  & strong right turn at 3.0 deg/s                                \\
$\tau$    & $s$ &  time until loss of vertical separation                            &        &                                \\
$a_{prev}$& ~ & previous advisory                            &        &                              \\ \hline
\end{tabular}}
\caption{Input variables used to produce a turn advisory.}
\label{tab:acasxu-inputs-actions}%
\end{table}

\subsection{Assumptions and Plant Model} Before we describe the plant model used in analysis, we first state our system assumptions: (i) the intruder flies in straight-line trajectories with constant speed, (ii) the ownship flies with constant speed and its heading is adjusted every second (the NNCS control period), (iii) the actions correspond to heading changes in the intruder of $1.5$ deg/sec for weak turn commands, $3.0$ deg/sec for strong turns and $0.0$ deg/sec for clear-of-conflict commands~\cite{julian2016policy}, (iv) there is no sensor noise and (v) advisories are followed exactly and immediately.
Many of these are fairly strong and the real system would need to be robust to maneuvering intruders, pilot delay and sensor noise.
From a safety proof perspective, however, we would want the system to \emph{at least} be safe under these ideal assumptions.

To model the state of the system with these assumptions, we use Cartesian coordinates.
The values $x_{own}, y_{own}, x_{int}, y_{int}$ refer to the $x$ and $y$ positions of the ownship and the intruder; $v_{own} = \sqrt{(v^x_{own})^2 + (v^y_{own})^2}$ and $v_{int} = \sqrt{(v^x_{int})^2 + (v^y_{int})^2}$ are the speed of the ownship and the intruder; $\theta_{own}$ and $\theta_{int}$ are the heading of the ownship and the intruder w.r.t the $x$ axis. 
The system performs idealized turn maneuvers modeled with Dubins aircraft dynamics:
\begin{equation}\label{eq:Dubins}
\begin{split}
    &\dot{x}_{own} = v^x_{own} = v_{own} cos(\theta_{own}) \\
    &\dot{y}_{own} = v^y_{own} = v_{own} sin(\theta_{own}) \\
    &\dot{x}_{int} = v^x_{int} = v_{int} cos(\theta_{int}) \\
    &\dot{y}_{int} = v^y_{int} = v_{int} sin(\theta_{int})
\end{split}
\end{equation}
%
%Note that all variables in Equation \ref{eq:Dubins} are measurable using sensors. 

Equation \ref{eq:Dubins} does not show clearly how the aircraft can be controlled by changing their heading. 
Taking derivatives of the Equation \ref{eq:Dubins} one more time and noticing that $\dot{\theta}_{own}$ is a constant between advisories, $\dot{\theta}_{own}  = (\pi/180)u = c (rad/s)$, and then taking $\dot{\theta}_{int} = 0$, we obtain the following 8-d linear system dynamics:
%
%\begin{equation}\label{eq:derivative}
%\begin{split}
%    &\dot{v}^x_{own} = -v_{own} sin(\theta_{own}) \dot{\theta}_{own} = -v_{own}^y \dot{\theta}_{own}, \\
%    &\dot{v}_{own}^y = v_{own} cos(\theta_{own}) \dot{\theta}_{own} = v_{own}^x \dot{\theta}_{own}, \\
%    &\dot{v}^x_{int} = -v_{int} sin(\theta_{int}) \dot{\theta}_{int} = -v_{int}^y \dot{\theta}_{int}, \\
%    &\dot{v}^y_{int} = v_{int} cos(\theta_{int}) \dot{\theta}_{int} = v_{int}^x \dot{\theta}_{int}, \\
%\end{split}
%\end{equation}
%

%
%Combining Equation \ref{eq:Dubins} and \ref{eq:derivative}, we can obtain a \emph{linear} model of the system with $8$ state variables at a time state $t$ as follows. 
%
\begin{equation} \label{eq:Model}
    \begin{bmatrix} \dot{x}_{own} \\ \dot{y}_{own} \\ \dot{v}^x_{own} \\ \dot{v}_{own}^y \\ \dot{x}_{int} \\ \dot{y}_{int} \\ \dot{v}^x_{int} \\ \dot{v}^y_{int} \\ \end{bmatrix} = \begin{bmatrix}
    0 & 0 & 1 & 0 & 0 & 0 & 0 & 0 \\ 0 & 0 & 0 & 1 & 0 & 0 & 0 & 0 \\ 0 & 0 & 0 & -c & 0 & 0 & 0 & 0 \\ 0 & 0 & c & 0 & 0 & 0 & 0 & 0 \\ 0 & 0 & 0 & 0 & 0 & 0 & 1 & 0 \\ 0 & 0 & 0 & 0 & 0 & 0 & 0 & 1 \\ 0 & 0 & 0 & 0 & 0 & 0 & 0 & 0 \\ 0 & 0 & 0 & 0 & 0 & 0 & 0 & 0 \\ 
    \end{bmatrix}
    \begin{bmatrix} {x}_{own} \\ {y}_{own} \\ {v}^x_{own} \\ {v}_{own}^y \\ {x}_{int} \\ {y}_{int} \\ {v}^x_{int} \\ {v}^y_{int} \\ \end{bmatrix} %+ 
    %\begin{bmatrix}
    %0 \\ 0 \\ 0 \\ 0 \\ 0 \\ 0 \\ 0 \\ 0 \\
    %\end{bmatrix}
    %u
\end{equation}
The linear model described in Equation \ref{eq:Model} is valid for only one control step, with a fixed control signal $u$, which may be either $-3, -1.5, 0, 1.5$ or $3$ deg/s depending on the specific command. 
%
%\textcolor{red}{(u should be in radians/sec.)}
%
Therefore, this model can be considered as a piece-wise linear model of the system. From the plant state variables, we can obtain the inputs for the neural network controller which are expected to in radial coordinates as follows. 
\begin{equation}\label{eq:lin_to_nonlinear}
    \begin{split}
        & \theta_{own} = arctan(\frac{v^y_{own}}{v^x_{own}}),~~~\theta_{int} = arctan(\frac{v^y_{int}}{v^x_{int}}), \\
        &\rho = \sqrt{(x_{int} - x_{own})^2 + (y_{int} - y_{own})^2}, \\
        & \theta = arctan(\frac{y_{int} - y_{own}}{x_{int} - x_{own}}) - \theta_{own}, ~~~ \psi = \theta_{int} - \theta_{own}.
    \end{split}
\end{equation}

\subsection{Reachability with $\mathcal{AH}$-Polytopes}
An $\mathcal{AH}$-polytope is a set representation that informally is an affine transformation of a half-space polytope, where the affine transformation and polytope terms are explicitly kept separate.
Although the name is fairly recent~\cite{sadraddini2019linear}, this set representation has often been used in reachability analysis for linear systems~\cite{bak2019hscc,bak17hscc} and neural networks~\cite{tran2019star,bak2020cav}, where it is also called a linear star set~\cite{duggirala2016parsimonious}, constrained zonotope~\cite{scott2016constrained}, affine form~\cite{han2006reachability}, or symbolic orthogonal projection~\cite{hagemann2014reachability}.

Importantly for this work, discrete-time reachability of systems with linear dynamics, $\dot{x} = Ax$, can be expressed exactly using this set representation, as it amounts to a linear transformation of the entire set by the matrix exponential $e^{At}$, where $t$ is the time step.
Further, operations like intersections can be performed exactly on $\mathcal{AH}$-polytopes, as well as linear optimization over the sets.

%A formal definition and operation list is provided in Appendix~\ref{apx:ahpoly}.
\begin{definition} [$\mathcal{AH}$-Polytope] \label{def:AH-polytope} An $\mathcal{AH}$-Polytope is a tuple $\Theta = \langle V, c, C, d \rangle$ that represents a set of states as follows:
\begin{equation*}
 \llbracket \Theta \rrbracket = \{x \in \mathbb{R}^n ~ | ~ \exists {\alpha \in \mathbb{R}^m}, ~ x = V \alpha + c \wedge C\alpha \leq d\}.   
\end{equation*}
\end{definition}
%We generally refer to both the tuple $\Theta$ and the set of states $\llbracket \Theta \rrbracket$ as $\Theta$ as it is clear from context which one we are referring to.
%
\begin{proposition}[Affine Mapping] An affine mapping of an $\mathcal{AH}$-Polytope $\Theta = \langle V, c, C, d \rangle$ with a mapping matrix $W$ and an offset vector $b$ is a new $\mathcal{AH}$-Polytope $\Theta^\prime = \langle V^\prime, c^\prime, C^\prime, d^\prime \rangle$ in which $V^\prime = WV,~c^\prime = Wc + b,~C^\prime = C,~d^\prime = d$.
\end{proposition}
\begin{proposition}[Linear Transformation] A linear transformation of an $\mathcal{AH}$-Polytope with a matrix $W$ is an affine mapping using mapping matrix $W$ and an offset vector of $b = 0$.
\end{proposition}
\begin{proposition}[Intersection] The intersection of $\Theta = \langle V, c, C, d \rangle$ and a half-space $\mathcal{H} = \{x~|~Gx \leq g \}$ is a new $\mathcal{AH}$-Polytope $\Theta^\prime = \langle V^\prime, c^\prime, C^\prime, d^\prime \rangle $ with $c^\prime = c,~V^\prime = V,~C^\prime = [C;GV],~d^\prime = [d;g-Gc]$.
\end{proposition}
\begin{proposition}[Linear Optimization] Linear optimization in given a direction $w \in \mathbb{R}^n$ over a star set
$\Theta = \langle V, c, C, d \rangle$
can be solved with linear programming as follows:
$\min(w^T x), ~s.t.~ x \in \Theta = 
w^T c + \min(w^T V \alpha),~s.t.~C\alpha \leq d$.
\end{proposition}

\subsection{Safety Problem Formulation}
Verifying the safety of the closed-loop system means proving the absence of \emph{unsafe paths} under all operating conditions.
For simplified presentation, we consider a discrete-time version of the problem, where we only check for collisions once a second when the system is activated.
Our analysis could be extended to continuous time through \emph{conservative time-discretization} approaches from hybrid systems reachability analysis~\cite{forets2021conservative}, which essentially bloat the initial set and then perform discrete-time analysis.
\begin{definition}[Path]
A path is written as $s_1 \xrightarrow{\alpha_1} s_2 \xrightarrow{\alpha_2} \ldots \xrightarrow{\alpha_{n-1}} s_n$, where successive values of $s_i$ and $s_{i+1}$ correspond to the state of the system one second apart according to the plant dynamics in Equation~\ref{eq:Model}. The command $\alpha_i$ is the system output from state $s_i$ using $\alpha_\text{prev} = \alpha_{i-1}$, with $s_1$ using the \textsc{coc} network.
Paths can either be in-plane, where $\dot{\tau} = 0$ and $\tau = 0$ in all states and so the $N_{1,*}$ networks get used to generate all commands, or out-of-plane, where $\dot{\tau} = -1$.
In the out-of-plane case, each state in the path should decrease $\tau$ by one second.
\end{definition}
An unsafe path has $s_1$ as an initial state and $s_n$ as an unsafe state.
\begin{definition}[Initial State]
An initial state of the state of the system is one where the aircraft are outside of the system's operating range ($\rho > 60760$ ft).
\label{def:initial_state}
\end{definition}

%The unsafe states are ones where the aircraft are too close together and a collision may occur.
%
\begin{definition}[Unsafe State]
Unsafe states are defined to be any states in the near mid-air collision (NMAC) cylinder~\cite{marston2015acas}, where the horizontal separation $\rho$ is less than 500 ft and the time to loss of vertical separation $\tau$ is zero seconds.
\label{def:unsafe_state}
\end{definition}
The operating conditions where the system should ensure safety are extracted based on the training ranges used for the original neural networks~\cite{kochenderfer2011robust,katz2017reluplex}.
The system should be active when the distance between aircraft $\rho \in [0, 60760]$ ft, otherwise clear-of-conflict is commanded.
The valid values for the ownship velocity are $v_{own} \in [100, 1200]$ ft/sec, valid values for intruder velocity are $v_{int} \in [0, 1200]$ ft/sec, and the angular inputs $\theta$ and $\psi$ are both between $-\pi$ and $\pi$.
%
%We make the following assumptions:

%discrete-time dynamics, once per second, can fix with approximation models

%no velocity changes

%straight-line trajectories for the intruder

%instant following of commands exactly at 1.5 deg / 3 deg / second

%clear-of-conflict means fly straight

%no sensor noise

%% file: algorithm.tex
\section{Quantized State Backreachability}
\label{sec:algorithm}
Our verification strategy is to compute the backwards reachable set of states from all possible unsafe states, trying to a find a path that begins with an initial state.
We first partition the unsafe states along state quantization boundaries.

\subsection{Partitioning the Unsafe States} 
Since the system advisories are only based on relative positions and headings, we eliminate symmetry by assuming that at the time of the collision the intruder is flying due east and at the origin.
We then consider all possible positions of the ownship to account for all possible unsafe states.
Three quantization parameters are used in the analysis: $q_\text{pos}$ to quantize positions, $q_\text{vel}$ to quantize velocities, and $q_\theta$ to quantize the heading angle.
Based on these parameters, we partition the unsafe states into 8-d $\mathcal{AH}$-polytopes covering the entire set of possible unsafe states.
The eight dimensions correspond to the system states in the linear dynamics in Equation~\ref{eq:Model}, including positions $x$, $y$, and velocities $v^x$, $v^y$ for both the ownship and intruder.
Associated with each partition, we also enumerate the five possible previous commands $\alpha_\text{prev}$ and two possibilities for whether there is a relative vertical velocity---whether the time to loss of vertical separation is fixed at 0 or decreasing, $\dot{\tau} \in \{0, -1\}$.

To create partitions, the $x_{own}$ and $y_{own}$ values are divided into a grid based on $q_\text{pos}$.
The intruder position $(x_{int}, y_{int})$ is set to $(0, 0)$.
The intruder and ownship velocities are partitioned based on $q_\text{vel}$, which gets reflected in the $x$ and $y$ velocity state variables for the two aircraft.
The intruder is moving due east, so $v_{int}^y = 0$ and $v_{int}^x$ is set to the range of intruder velocities corresponding to the current partition.
The heading of the angle of the ownship is partitioned based on $q_\theta$, where each partition has a lower and upper bound on the heading $[\theta_{own}^{lb}, \theta_{own}^{ub}]$. 
From the current range of values for the ownship heading and the range of values for the ownship velocity, we can construct linear bounds on $v_{own}^x$ and $v_{own}^y$.
This is done by connecting five points, $a, b, c, d$ and $e$, where $a$ and $b$ are the points at two extreme angles and minimum velocity, $c$ and $d$ are the two extreme angles and max velocity, and $e$ is the point at the intersection of the tangent lines of the maximum velocity circle at $c$ and $d$.
A visualization is shown in Figure~\ref{fig:vel_partition}.
We generally use $q_\theta = 1.5$ deg (as it makes for a cleaner backreachability step), which guarantee all possible $v_{own}^x, v_{own}^y$ values are covered.

\begin{figure}[t]
    \centering
    \includegraphics[width=0.6\columnwidth]{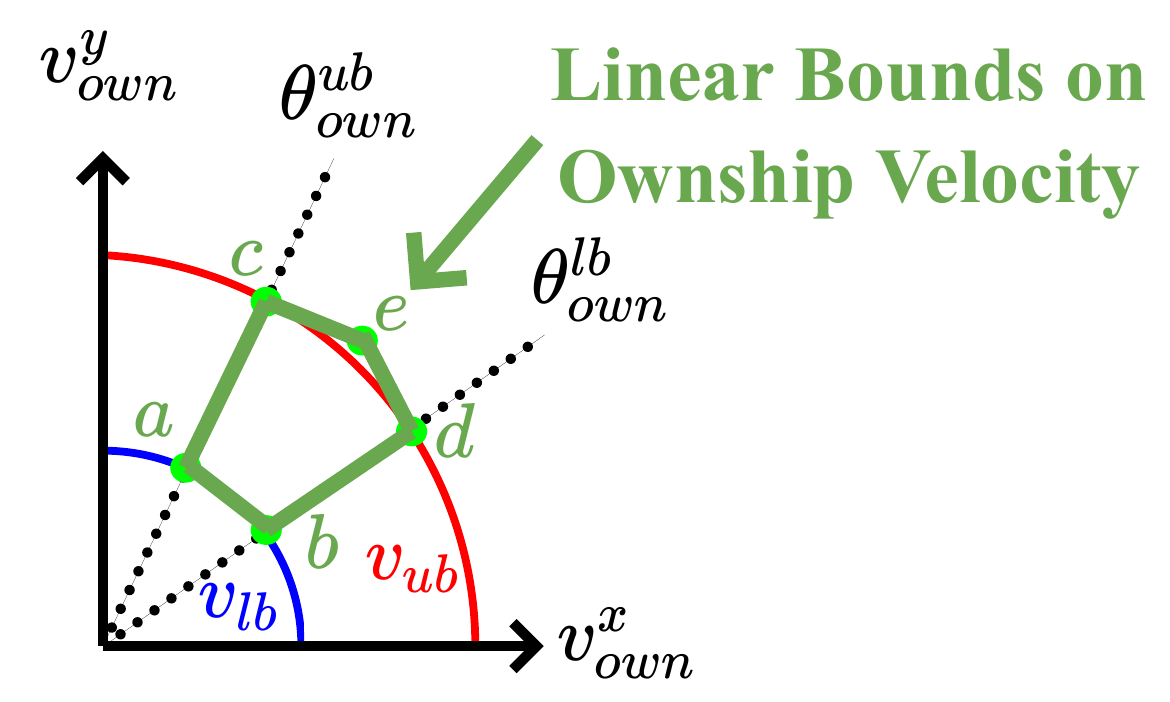}
    \caption{The ownship velocity range and heading angle range are used to create linear bounds on $v_{own}^x$ and $v_{own}^y$ by connecting the points $a, b, c, d$ and $e$.}
    \label{fig:vel_partition}
\end{figure}

\subsection{Backreachability from Each Partition} 
Once a covering of the entire set of unsafe states is performed, for each partition we compute the \emph{exact} set of predecessor states that can lead to the states in the partition at a previous step.
This process is repeated until either no predecessors exist or an initial state predecessor is found\footnote{Degenerate paths could theoretically exist of infinite length that never include a valid initial state, but we did not observe this occurring in practice.}, as described in Definition~\ref{def:initial_state}.
In the latter case, a path exists from an initial state to a partition of the unsafe states in the quantized closed-loop system.
Otherwise, if no partitions contain unsafe paths, then the quantized closed-loop system is safe.

The \texttt{check\_state} function in Algorithm~\ref{alg:highlevel} recursively computes and checks predecessors.
The input is a state set $\mathcal{S}$, which is initially an 8-d partition of the unsafe states represented as an $\mathcal{AH}$-Polytope, as well as the associated value of $\alpha_\text{prev}$ and the time to loss of vertical separation, $\tau = 0$ in all unsafe states.

%%%%%%%%%%%
\input{pseudocode}
%%%%%%%%%%5

In line~\ref{line:prev}, \texttt{backreach\_step} is called, which returns the predecessor set of states as an $\mathcal{AH}$-polytope $\mathcal{P}$.
This is done by taking the linear derivative matrix $A_c$ from Equation~\ref{eq:Model} with the value of $c$ corresponding to $\alpha_\text{prev}$, and then computing the matrix exponential $W = e^{-A_c}$.
The resulting matrix is the solution matrix for the system one second prior.
A linear transformation of the $\mathcal{AH}$-polytope $S$ is then performed by $W$ in order to obtain $\mathcal{P}$.
In line~\ref{line:tau_prev}, the value of the time to loss of vertical separation at the previous step $\tau_\text{prev}$ is computed.
This either always equals 0 if $\dot{\tau} = 0$ for the current partition  corresponding to in-plane flight, or increases by 1 at each call to \texttt{check\_state} if $\dot{\tau} = -1$  for out-of-plane flight.

Next, the algorithm computes states in $\mathcal{P}$ where the command produced by the networks was $\alpha_\text{prev}$ and the time to loss of vertical separation was the value at the previous step, $\tau_\text{prev}$.
This requires iterating over the five possible networks that could have been used at the prior state (the loop on line~\ref{line:prevprevloop}).
For each network (corresponding to $\alpha_\text{prevprev}$), we check each quantized state in $\mathcal{P}$ (line~\ref{line:possible_quantized}).

% (dx, dy, q_theta1, s.qv_own, s.qv_int)
The \texttt{possible\_quantized\_states} returns a list of \emph{quantized states}, which are 5-tuples of integers, 
$q=(dx, dy, \theta_{own}, v_{own}, v_{int})$.
The $dx$ and $dy$ terms correspond to the difference in positions between the intruder and ownship, divided by the position quantum $q_\text{pos}$.
The $\theta_{own}$ term is the heading angle divided by $q_\theta$, and the velocities $v_{own}$ and $v_{int}$ are the fixed aircraft velocities, integer divided by $q_\text{vel}$.
The function computes the possible quantized states by using linear programming to find $\mathcal{P}$'s bounding box, and then looping over possible quantized states to check for feasibility when intersected with $\mathcal{AH}$-polytope $\mathcal{P}$.

Line~\ref{line:run_network} runs the neural network corresponding to $\alpha_\text{prevprev}$ on quantized state $q$ to check if the correct command ($\alpha_\text{prev}$) is obtained.
This process requires converting from the quantized state (a 5-tuple of integers) to continuous inputs for the neural network.
To do this, we use Equation~\ref{eq:lin_to_nonlinear}, noting that the $\theta_{own}$ is quantized using $q_\theta$, $\theta_{int}$ is always 0, and the computation of $\rho$ and $\theta$ uses the dequantized value of $dx$ and $dy$ ($x_{int} - x_{own}$ is taken to be $\frac{q_\text{pos}}{2} + dx * q_\text{pos}$).

When the network output matches the required $\alpha_\text{prev}$ command, line~\ref{line:append} adds the quantized state to the valid list of predecessors \texttt{predecessor\_quanta}. 
Otherwise, line~\ref{line:all_correct_false} sets the \texttt{all\_correct} flag to \texttt{false}, since some of the quantized states are not valid predecessors.
Line~\ref{line:unsafe} checks if the predecessor state satisfies the initial state condition, in which case an unsafe path has been found.
On this line, $\rho_\text{min}(q)$ is the minimum aircraft separation distance in the quantized state $q$, which must be greater than 60760 ft in an initial state.

After classifying each quantized predecessor state, either all quantized states had the correct output or some did not.
Based on this, we either recursively call \texttt{check\_state} on the entire set $\mathcal{P}$ (line~\ref{line:single_recurse}), or we split the set $\mathcal{P}$ into parts, and only recursively call \texttt{check\_state} on parts that had the correct output.
On line~\ref{line:quantized_to_state_set}, \texttt{quantized\_to\_state\_set} returns the 8-d continuous states corresponding to the quantized state $q$, which is then intersected with $\mathcal{P}$ before being recursively passed to \texttt{check\_state}.
When splitting is performed, it is possible that no states had the correct output (\texttt{predecessor\_quanta} may be empty).

An illustration of the algorithm is provided in Figure~\ref{fig:check-state-illustration}.
In the figure, the set $\mathcal{P}$ is covered by nine quantized states returned by \texttt{possible\_quantized\_states} (the dots on the right side).
Of these nine, eight have a correct output (blue dots), and one has an incorrect output (red dot).
In this case, the algorithm would split the set $\mathcal{P}$ into eight parts and call \texttt{check\_state} on each recursively\footnote{An implementation optimization could be to reduce this splitting into only three parts. Three is the minimum in this case, since $\mathcal{AH}$-polytopes must be convex.}.

\begin{figure}[t]
    \centering
    \includegraphics[width=\columnwidth]{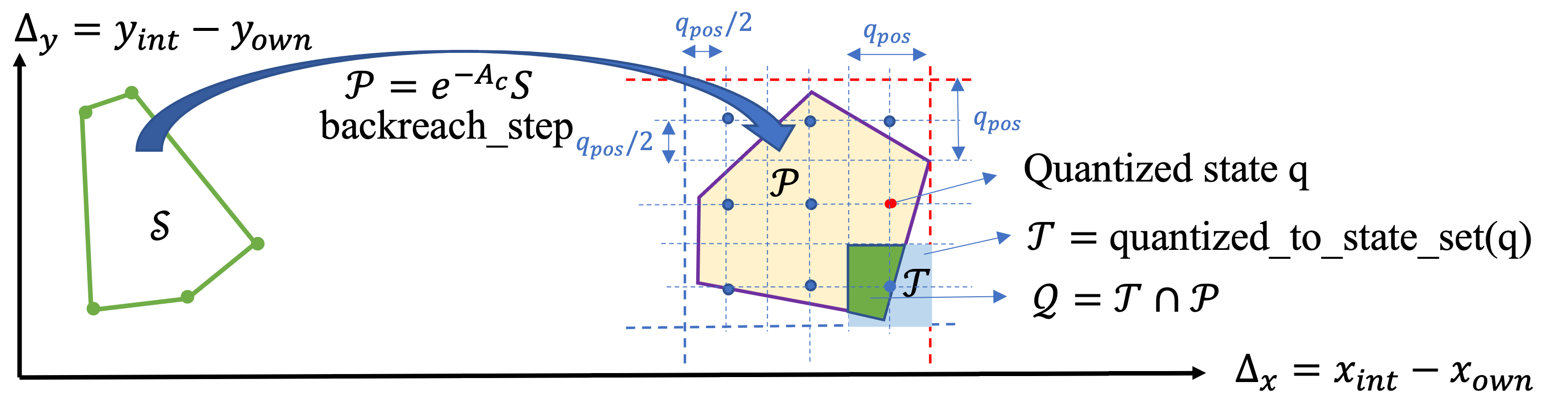}
    \caption{Illustration of Algorithm~\ref{alg:highlevel} given state set $\mathcal{S}$.}
    \label{fig:check-state-illustration}
\end{figure}

We next prove the described algorithm is sound with respect to the safety of the quantized closed-loop system.

\begin{theorem}[Soundness]
If \textnormal{\texttt{check\_state}} returns \textnormal{\textsf{safe}} for every partition, the quantized closed-loop system is safe.
\end{theorem}
\begin{proof}
We proceed by contraction. 
Assume the quantized closed-loop system is unsafe and so these exists a finite path from an initial state to an unsafe state, $s_1 \xrightarrow{\alpha_1} s_2 \xrightarrow{\alpha_2} \ldots \xrightarrow{\alpha_{n-1}} s_n$.
%\alpha_2, \ldots, \alpha_{n-1}$.
%
% line:possible_quantized
Since the unsafe state partitioning covers the full set of unsafe states, the unsafe state $s_n$ is in some partition.
We can follow the progress of $s_n \in \mathcal{S}$, through \texttt{check\_state} at each recursive call.

At each call, $s_i \in \mathcal{S}$ has a predecessor $s_{i-1} \in \mathcal{P}$ that gets to $s_i$ using command $\alpha_{i-1}$.
In the call to \texttt{check\_state}, $\alpha_\text{prev}$ will be $\alpha_{i-1}$.
The value of $\tau_\text{prev}$ is incremented at each call on line~\ref{line:tau_prev} and so always correctly corresponds to $s_{i-1}$.
%
%The loop at line~\ref{line:prevprevloop} checks all values of $\alpha_\text{prevprev}$ and including $\alpha_{i-2}$.
%
Since $s_{i-1} \in \mathcal{P}$, $s_{i-1}$ will also be in one of the quantized states $q_{i-1}$ checked on line~\ref{line:possible_quantized}.
The existence of the counterexample path segment $\xrightarrow{\alpha_{i-2}} s_{i-1} \xrightarrow{\alpha_{i-1}} s_i$ means that the condition on line~\ref{line:run_network} will be true when $\alpha_\text{prevprev} = \alpha_{i-2}$, and so $q_{i-1}$ will be added to \texttt{predecessor\_quanta}.
Since $s_{i-1}$ is both in $\mathcal{P}$ and in the state set corresponding to a quantized state in \texttt{predecessor\_quanta}, it will be used in a recursive call to \texttt{check\_state}.
This argument can be repeated for all states in the unsafe path back to the initial state $s_1$, which would have been returned as \textsf{unsafe} on line~\ref{line:unsafe} rather than used in a recursive call.
This contradicts the assumption that \texttt{check\_state} returned \textsf{safe} for every partition. \qed
\end{proof}

%%%%%%%%%%%%%%%%%%%%%%%%%%%%%%%%%%%%%%%%%%%%%%%%%%%%%%55
\subsection{Falsification of Original (Unquantized) System}
\label{ssec:falsification}
The algorithm in the previous section can be used to efficiently find unsafe paths of the original, unquantized, closed-loop neural network control system.
This is done by repeatedly calling the algorithm with smaller and smaller quantization constants $q_\text{pos}$, $q_\text{vel}$ and $q_\theta$ and checking the quantized system for safety.

At each step if the safety proof fails, with small modifications to \texttt{check\_state} we can get the trace corresponding to the unsafe path for each partition.
In particular, rather than simply returning \textsf{unsafe} on line~\ref{line:unsafe}, we can instead return the set of unsafe initial states \texttt{quantized\_to\_state\_set}(q) $\cap ~ \mathcal{P}$.
A witness point inside this set can be obtained through linear programming\footnote{For witness points, we use the Chebyshev center of the six-dimensional state polytope (removing $y_{int}$ and $v^y_{int}$ since they are fixed at zero), as it helps avoid numerical issues that can occur at the boundaries of the set.}.
This witness point is then executed on the original system, without quantization, checking for safety.
If the witness point is safe in the non-quantized system, the quantization constants are refined by taking turns dividing each of them in half.

\begin{theorem}[Completeness]
By following the falsification approach above and repeatedly refining $q_\text{pos}$, $q_\text{vel}$ and $q_\theta$, either we will prove the quantized system is safe or find an unsafe trace in the original, unquantized system.
\end{theorem}
\begin{proof}
First, consider the case that the system is \emph{robustly unsafe}, which we define as there existing a ball $\mathcal{B}_\text{init}$ of initial states of radius $\delta > 0$ that all follow the same command sequence $\alpha_1, \alpha_2, \ldots, \alpha_n$ and end in the unsafe set.
%
%In this case, there must also be a ball in the unsafe states of some other radius $\delta' > 0$ of entirely unsafe states, $\mathcal{B}_\text{unsafe}$.
%
Since all the initial states follow the same command sequence, the linear transformations corresponding to the commands $\alpha_1, \alpha_2, \ldots, \alpha_n$, which we call $A_{c_1}, A_{c_2}, \ldots, A_{c_n}$
can be multiplied together into a single matrix that transforms initial states to unsafe states, $A_{C} = A_{c_n} \ldots A_{c_2} A_{c_1}$.
The matrix $A_{C}$ is invertible since all the transformations corresponding to each command $A_{c_1}, A_{c_2}, \ldots, A_{c_n}$ are invertible.
The matrix $A_{C}$ being invertible means that since the volume of the ball in the initial states $\mathcal{B}_\text{init}$ is nonzero, the corresponding set of states in the unsafe set is an ellipsoid with nonzero volume, which we call $\mathcal{E}_\text{unsafe}$.
Through refinement of the quantization parameters $q_\text{pos}$, $q_\text{vel}$ and $q_\theta$, eventually a partition will be entirely contained in $\mathcal{E}_\text{unsafe}$.
When this happens, every witness point of the quantized counterexample from that partition will be in $\mathcal{B}_\text{init}$, and so will be an initial state of an unsafe oath of the original, unquantized system.

Perhaps less practically, even if the original system is not robustly unsafe, the process still will theoretically terminate when finite-precision numbers are used in the non-quantized system, such as with air-to-air collision avoidance neural networks that use 32-bit floats.
As the quantization values are halved, the difference between the unsafe state in the quantized and nonquantized system is also reduced, until it reaches numeric precision. \qed
\end{proof}

The second case may seem like one needs to split the entire state space up to machine precision, which would make it very impractical. 
However, if the goal is to search for counterexamples, then the process can first refine the regions that were found as unsafe using the previous quantization values, in a depth-first search manner.
In this way, when the system is unsafe the process would not need to immediately refine the entire state space in order to find these counterexamples.
Also keep in mind that the quantized system being safe is a valid outcome of this refinement process, and this does not mean that the original, unquantized system, is safe.

%% file: pseudocode.tex
\begin{algorithm}[t]
\SetKwInOut{Function}{Function}
\SetKwInOut{Input}{Input}
\SetKwInOut{Output}{Output}
\SetKw{Return}{return}
\Function{\texttt{check\_state}, Recursively checks safety of predecessors}
\Input{State set: $\mathcal{S}$, Prev cmd: $\alpha_{\textnormal{prev}}$,
Time to loss of vertical separation: $\tau$}
\Output{Verification Result (\textsf{safe} or \textsf{unsafe})}

$\mathcal{P}$ = \texttt{backreach\_step}($\mathcal{S}$, $\alpha_\text{prev}$) \tcp{state set of one-step predecessors}  \label{line:prev}

$\tau_\text{prev} = \tau - \dot{\tau}$ \tcp{$\dot{\tau}$ is fixed at either 0 or -1} \label{line:tau_prev}

\For{ $\alpha_{\textnormal{prevprev}}$ \textnormal{ in } \textsc{[coc, wl, wr, sl, sr]} \label{line:prevprevloop}
}{
predecessor\_quanta $\gets$ List()

all\_correct $\gets$ \textsc{true}

\For{ $q$ \textnormal{ in } $\textnormal{\texttt{possible\_quantized\_states}}(\mathcal{P})$ \label{line:possible_quantized}
}{
\uIf{\textnormal{\texttt{run\_network}(}$\alpha_{\textnormal{prevprev}}, \tau_\text{prev}$, q\textnormal{)} = $\alpha_\textnormal{prev}$ \label{line:run_network}
}{
predecessor\_quanta.\texttt{append}($q$) \label{line:append}

\uIf{$\rho_\text{min}(q) > 60760$}{
  \Return \textsf{unsafe} \tcp{predecessor is valid initial state} \label{line:unsafe}
  }
} % end if
\uElse{
all\_correct $\gets$ \textsc{false} \label{line:all_correct_false}
}% end else
} % end for

%%%%%%%%%%%%%%%%%%%%%%%%%%
\uIf{\textnormal{all\_correct}}{
\tcp{recursive case without splitting}

  \If{\textnormal{\texttt{check\_state}}($\mathcal{P}, \alpha_{\text{prevprev}}, \tau_\text{prev}$) = \textnormal{\textsf{unsafe}} \label{line:single_recurse}}
  {\Return{\textnormal{\textsf{unsafe}}}}
  
} % end if
\uElse{
\tcp{recursive case with splitting along quantum boundaries}

  \For{ $q$ \textnormal{ in } $\textnormal{predecessor\_quanta}$
    }
    {
    $\mathcal{T} \gets $ \texttt{ quantized\_to\_state\_set}(q) \label{line:quantized_to_state_set}
    
    $\mathcal{Q} \gets \mathcal{T} \cap \mathcal{P}$ \label{line:quantized_intersect}
    
    \If{\textnormal{\texttt{check\_state}}($\mathcal{Q}, \alpha_{\text{prevprev}}, \tau_\text{prev}$) = \textnormal{\textsf{unsafe}}}
  {\Return{\textnormal{\textsf{unsafe}}}}
    }
} % end else

} % end for alpha_prevprev
\Return{\textnormal{\textsf{safe}}}

\caption{High-level algorithm for single partition backreachability.}
\label{alg:highlevel}
\end{algorithm}

%% file: results.tex
\section{Evaluation}
We implemented the approach and set out to prove the safety of quantized closed-loop air-to-air collision avoidance system\footnote{The code and instructions to reproduce all the results are online: \url{https://github.com/stanleybak/quantized_nn_backreach/releases/tag/NFM2022_submitted}}.
%
%Our implementation was created in Python using the $\mathcal{AH}$-Polytope library from nnenum~\cite{bak2021nnenum}, which uses GLPK for linear programming, and the \texttt{onnxruntime} library for neural network execution.
%
We ran the measurements on an Amazon Web Services (AWS) Elastic Computing Cloud (EC2) server with a \texttt{c6i.metal} instance type, which has a 3.5 GHz Intel Xeon processor with 128 virtual CPUs, and 256 GB memory.
The algorithm is easily parallelized as proofs for each partition of the unsafe states can be checked independently.

\subsection{Complete Proof of Safety Attempt}
We first attempted a proof of safety for the entire range of unsafe states for ACAS Xu.
For this, we started with large quantization values, $q_\text{pos}$ = 500 ft, $q_\text{vel}$ = 100 ft/sec, and $q_\theta$ = 1.5 deg.
In this case, the unsafe near-mid-air collision circle of radius 500 ft can be covered with 4 partitions, the complete velocity range of the ownship $[100, 1200]$ needs 11 partitions, the velocity of the intruder $[0, 1200]$ needs 12 partitions, the heading angle of the ownship is divided into $\frac{360\text{ deg}}{1.5\text{ deg}}$ = 240 partitions, and there are 5 choices for the $\alpha_\text{prev}$ and two possibilities to check for $\dot{\tau}$.
Multiplying these together, we get a total of 1267200 partitions of the unsafe states, each of which we pass to \texttt{check\_state} (Algorithm~\ref{alg:highlevel}).

%
% pos_q = 500 #250
% vel_q = 100 # 50
% theta1_q_deg = 1.5
%
This quickly, within a minute, finds counterexamples in the quantized system. 
When the witness initial states of the quantized counterexample are replayed on the original non-quantized system, according to the falsification algorithm from Section~\ref{ssec:falsification}, these were also found to be unsafe!
The exact runtime before an unsafe case is found depends on the order in which the partitions are searched, but we found that although changing this did affect the counterexample produced, the runtime was usually less than a minute.
Two of the unsafe cases are shown in Figure~\ref{fig:unsafe_square} in parts (a) and (b).

In the situation shown in Figure~\ref{fig:unsafe_square}(a), the intruder starts beyond the range of the network ($\rho > 60780$ ft).
As soon as the intruder gets in range, a turn is commanded, but the velocity of the ownship is slow and a collision still occurs.
This situation looks like it could be fixed by increasing the range of the system beyond $60780$ ft---likely requiring retraining the networks---to allow a turn to be commanded earlier.
Alternatively, perhaps adding a ``do not turn'' option as a possible output would be another way to address this scenario (clear-of-conflict could allow the ownship to maneuver as desired which may be unsafe here).

Figure~\ref{fig:unsafe_square}(b) shows another unsafe case found that is particularly concerning.
This is a tail-chase scenario, although the ownship is already moving away from the straight-line trajectory of the intruder.
The system nonetheless commands a turn and actively maneuvers the ownship aircraft back into the path of the intruder.
This situation demonstrates one of the dangers of the collision risk metric used to evaluate the effectiveness of many air-to-air collision avoidance systems, which compares the number of near mid-air collisions (NMAC) with and without the system using a large number of simulations.
Although a system can be effective by this metric, in specific cases it may still create collisions that would not otherwise have occurred, as demonstrated in this scenario.

\begin{figure}[t]
\centering
\subfloat[Immediate Turn Command. \\ Video: \url{https://youtu.be/WZWPKwIXfcM}]
{\includegraphics[width=.48\textwidth]{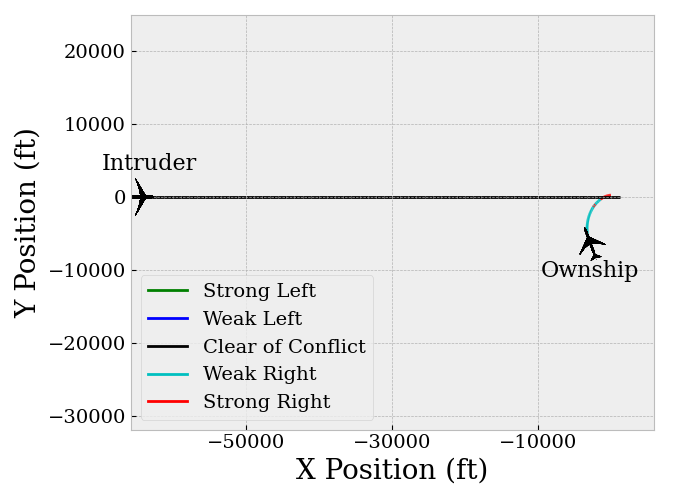}}
~~
\subfloat[Safety System Causes Crash. \\ Video: \url{https://youtu.be/dDwRiv_Kh2M} ]
{\includegraphics[width=.48\textwidth]{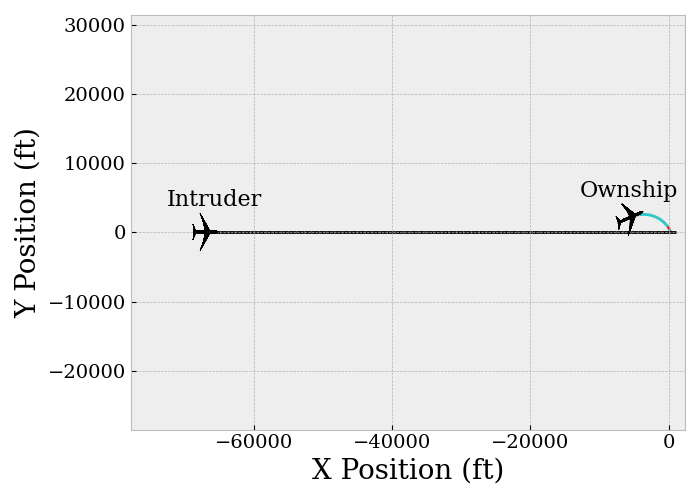}}
\\
\subfloat[Fast Ownship and $\tau > 0$. \\ Video: \url{https://youtu.be/F_bykLR9lJw}]
{\includegraphics[width=.48\textwidth]{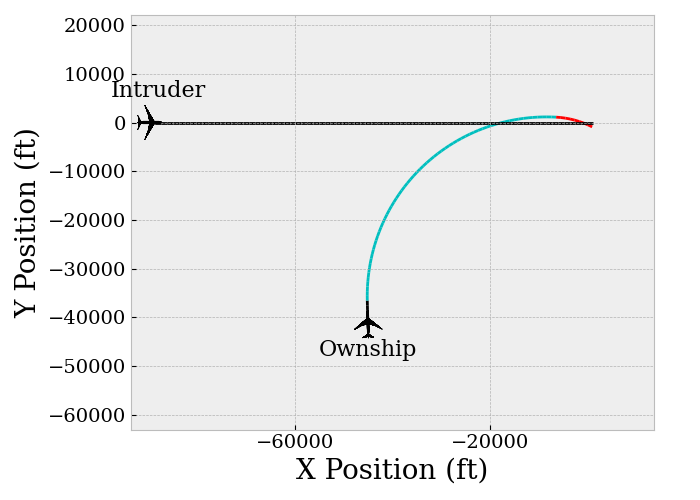}}
~~
\subfloat[Slow Intruder. \\ Video: \url{https://youtu.be/7B_-k0qpZTo}]
{\includegraphics[width=.48\textwidth]{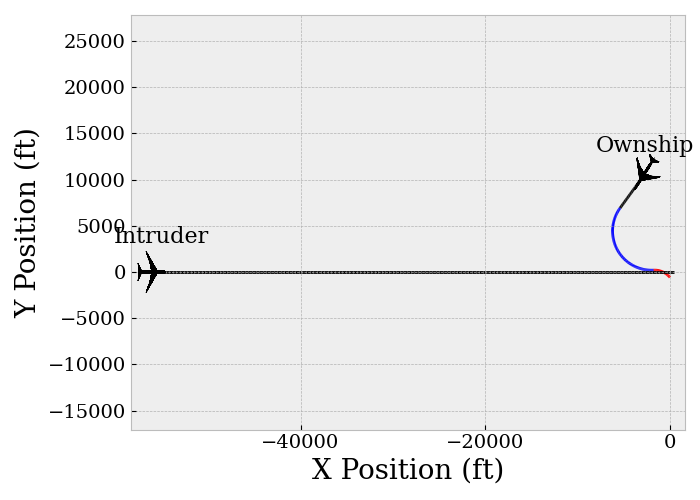}}
\caption{Unsafe counterexamples found in the original non-quantized NNCS. 
\ifarxiv
Detailed traces of the counterexamples are provided in Appendix~\ref{sec:counterexamples}.
\else
Step-by-step traces of the counterexamples are provided in the appendix of the extended report: \url{https://arxiv.org/abs/2201.06626}.
\fi
}\label{fig:unsafe_square}
\end{figure}

\subsection{Proving Safety in More Limited Operating Conditions}
\label{ssec:restricted_op}
As the proof of safety for the entire operating range failed, we next tried to prove safety in restricted operating conditions.
Many of the unsafe situations found, including the two above, had a slow ownship velocity and a fast intruder.
By making the ownship fast enough, we hypothesized collisions could be avoided.

When we restricted the range of $v_{own}$ to be in $[1000, 1200]$ ft/sec, using $q_\text{pos}$ = 250 ft, $q_\text{vel}$ = 50 ft/sec, and $q_\theta$ = 1.5 deg, we were able to guarantee safety of the quantized closed-loop neural network control system.
The proof requred checking 3.7 million cases and took about 32 minutes.
The longest runtime for any single call to \texttt{check\_state} (checking a single partition) was 63 seconds.

Reducing $v_{own}$ further to [950, 1000] ft/sec made the quantized system unsafe. 
Following the falsification approach from Section~\ref{ssec:falsification}, we refined the quantization parameters until we were able to find a counterexample in the original unquantized closed loop system.
In this case, the ownship was moving with $v_{own} = 964.1$ ft/sec, and the time to loss of vertical separation $\tau$ was initially 75 secs (the quantized system was safe for in-plane flight, with $\dot{\tau} = 0$).
This case is shown in Figure~\ref{fig:unsafe_square}(c).

From the other side, we can alternatively attempt to prove safety under the assumption that the intruder is slow without restricting the ownship's velocity.
In this case, the method also finds unsafe counterexamples in the unquantized system, such as the 159 second trace shown in Figure~\ref{fig:unsafe_square}(d) with $v_{int} = 390.1$ ft/sec.
\ifarxiv
The full trace for this situation is provided in Appendix~\ref{apx:slow_intruder} and has a peculiar characteristic.
\else
The full trace for this situation is provided in the appendix of the extended report\footnote{\url{https://arxiv.org/abs/2201.06626}} and has a peculiar characteristic.
\fi
The command switch from weak-left to strong-right a few seconds before the collision corresponds to the relative position angle $\theta$ wrapping from $-\pi$ to $\pi$.
This discontinuity in the network input between successive steps is a strong candidate root cause of the eventual near mid-air collision.

\subsection{Comparison with Other Approaches}
\label{ssec:comparison}
As far as we are aware, the proposed method is the first to provide safety guarantees while varying all of the operating conditions of the neural network compression of the collision avoidance system.

One related technique, based on computing discrete abstractions and forward reachability was able to provide safety guarantees for the similar Horizontal CAS~\cite{julian2019guaranteeing}.
This system is simpler to analyze: the inputs were modified to take in Cartesian state variables, the operating range was smaller ($\rho < 50000$), there were  fewer neural networks in the system, each of which had half as many neurons per layer, and critically, fixed velocities of $v_{own}=200$ and $v_{int}=185$ were considered, rather than using velocity ranges.
Despite these simplifications, analysis took 227 CPU hours, mostly on the neural network analysis step to analyze 74 million partitions.
For a comparison, we analyzed the larger neural networks in this work with the proposed state quantization and backreachability method, using the same fixed $v_{own}$ and $v_{int}$ values. 
Using a quantized system with $q_{pos} = 250$ ft and $q_\theta$ = 1.5 deg, the method proved safety of all 38400 partitions of the unsafe states in 60.6 seconds.
Also note that while the Horizontal CAS discrete abstraction approach can sometimes prove safety, it would be poor at generating counterexamples, as abstract reachability overapproximates the true reachable set; abstract counterexamples do not correspond to real counterexamples. 
In contrast, the backwards reachability performed in Algorithm~\ref{alg:highlevel} is exact with respect to the quantized system, and the gap between the quantized and original system can be reduced by refining the quantization parameters, making it highly effective for counterexample generation.

We also compared our method with simulation-based analysis, which cannot provide guarantees about system safety but should be able to find unsafe counterexamples if enough simulations are attempted, as the system was shown to be unsafe.
In earlier work~\cite{julian2019deep}, 1.5 million encounters were simulated for the original neural network compression to evaluate the risk of collisions, sampling from probability distributions of actual maneuvers and taking into account sensor noise.
We evaluated the same number of simulations without sensor noise and sampling over the entire set of operating conditions, in order to match the assumptions used in the safety proof.
We generated uniform random initial states by considering an initial $\rho \in [60760, 63160]$ and $\theta, \psi, v_{own}$ and $v_{int}$ in their entire operating range.
When considering $\dot{\tau} = -1$, we assigned the initial value of $\tau$ between 25 and 160 seconds, as the unsafe case in Figure~\ref{fig:unsafe_square}(d) was a 159 second trace.
We repeated the process of running 1.5 million simulations one hundred times each for both $\dot{\tau} = -1$ and $\dot{\tau} = 0$, in order to account for statistical noise.

In the $\dot{\tau} = 0$ case, each batch of 1.5 million simulations found on average 17.07 unsafe paths.
The unsafe cases were dominated by situations where the intruder velocity was low and the ownship velocity was high.
The mean value of $v_{int}$ was 997.8, with a standard deviation of 147.5.
The lowest values of $v_{int}$ over the unsafe cases in all 150 million simulations was 927.6, whereas Figure~\ref{fig:unsafe_square}(d) showed a case with $v_{int} =$ 390.1 found with our approach.
The mean value of $v_{own}$ in the unsafe cases was 133.4 with a standard deviation of 43.0.
The greatest value of $v_{own}$ over all the unsafe cases found with 150 million simulations was 452.3, whereas our approach found an in-plane case with $v_{own} =$ 881.6.

The performance of simulation analysis for the out-of-plane case is even worse, as the initial state must also correctly choose the value of the time to loss of vertical separation $\tau$ in order to find a collision.
Each batch of 1.5 million simulations with $\dot{\tau} = -1$ had on average 0.07 unsafe simulations.
The maximum ownship velocity $v_{own}$ in the unsafe cases had a mean of $175.4$ with a standard deviation of $77.9$.
The greatest value of $v_{own}$ over the unsafe cases found in all 150 million simulations was 343.0, whereas our approach found a case with $v_{own}=$ 964.1, as shown before in Figure~\ref{fig:unsafe_square}(c).

Overall, while simulation analysis may find some unsafe cases, it would be difficult to find the extreme velocity cases discovered with the proposed approach.
Further, simulation analysis is incomplete and cannot prove safety for the system under subsets of operating conditions as was done in Section~\ref{ssec:restricted_op}.
%\subsection{Out-of-Plane Case}

%\subsection{Comparison with Simulation-Base Analysis}

%% file: related.tex
\section{Related Work}
\textbf{Simulation-based Safety Analysis.} The air-to-air collision avoidance system was originally evaluated using 1.5 million simulations \cite{kochenderfer2010airspace} based on Bayesian statistical encounter models.
This uses relaxed assumptions compared with our work, such as allowing for changes in acceleration.
The output of such analysis is not a yes/no assessment of safety, as the system can clearly be unsafe if the intruder is faster than the ownship and maneuvers adversarially, but rather a risk score assessment of the change in safety compared to without using the system.
Via simulation, given a bounded uncertainty in sensing and control, the probability of near-mid-air-collision was about $10^{-4}$ \cite{julian2019deep}.
Although simulations show that the system may be unsafe, we do not know if the collision occurs due to the uncertainty or the system itself.
In this work, we could show that the system itself was unsafe, even if we have perfect sensing and control. %

\vspace{1em}
\noindent
\textbf{Verification of NNCS.} 
The Verisig approach \cite{ivanov2019verisig} verifies a NNCS by transforming a network with a sigmoidal neural network controller to an equivalent hybrid system that can be analyzed with Flow*~\cite{chen2013flow}, a well-known tool for verifying nonlinear hybrid systems.
Another method \cite{huang2019reachnn,dutta2019reachability} combines polynomial approximation of the neural network controller with the plant's physical dynamics to construct a tight overapproximation of the system's reachable set.
The star set approach \cite{tran2019safety} shows that the exact reachable set of an NNCS with a linear plant model and a ReLU neural network controller can be computed, although this is expensive when initial states are large.
These methods build upon open-loop neural network verification algorithms~\cite{liu2019algorithms,tran2020verification}, which can be difficult to scale to large complex networks~\cite{vnncomp2021} and can sometimes lose soundness due to floating-point numeric issues~\cite{zombori2020fooling}.
The proposed quantization analysis only needs to execute neural networks, and so does not suffer from these problems.

\vspace{1em}
\noindent
\textbf{Verification of the Closed-loop Air-to-Air Collision Avoidance System.}
%
%Recent efforts \cite{julian2019guaranteeing,lopez2021scitech,claviere2021safety} use forward reachability methods to verify the safety of the closed-loop ACAS Xu System.
%
Existing works have verified NNCS with a single neural network controller on a small set of initial states \cite{johnson2021arch}. 
The closed-loop system involves switching between multiple neural networks and has a large set of initial states, creating a unique challenge for verification.
The simplified Horizontal CAS system was analyzed using fast symbolic interval analysis for neural network controllers~\cite{wang2018formal} to construct a discrete abstraction~\cite{julian2019guaranteeing}.
This method can consider sensor uncertainty, inexact turn commands, and pilot delay, although simplified assumptions are made, as discussed in Section~\ref{ssec:comparison}.
Recently, the same system as this work has been verified with extensions of the symbolic interval method~\cite{claviere2021safety} and with star-based reachability \cite{lopez2021scitech} in nnv~\cite{tran2020cav2} and nnenum~\cite{bak2021nnenum}.
These approaches use forward reachability analysis and provide sound but not complete verification results. 
However, verification has only been demonstrated for specific scenarios with small sets of initial states, not the full operating conditions considered here.

%% file: conclusion.tex
\section{Conclusion}
In this work, we set out to prove the \emph{closed-loop} safety of one of the most popular benchmarks for neural network verification methods, using a new algorithm based on state quantization and backreachability.
In principle, the approach scaled sufficiently well to be able to verify the system under all valid initial states and aircraft velocities.
However, the proof process instead found many unsafe scenarios where the original, unquantized system had near mid-air collisions, despite ideal assumptions on sensors and maneuvering.
Compared with random simulation-based analysis, we could find counterexamples at more extreme velocities, as well as provide proofs of safety of the quantized closed-loop system in more limited scenarios.

The approach is could be attractive for certification. 
A system with a quantization layer behaves like a large lookup table, and the method is therefore effective on any size network with any layer type, and may even be applicable to other machine learning approaches.
The trade-off of quantization is usually a small degradation in performance of the controller, with a significant benefit of reducing analysis complexity and allowing for the possibility of verification.

%% file: apx_counterexamples.tex
\section{Unsafe Counterexample Details}
\label{sec:counterexamples}
Here, we provide detailed traces on the unsafe counterexamples found in the original unquantized neural network control system shown previously in Figure~\ref{fig:unsafe_square}.

\input{apx_first}
\input{apx_causecrash}
\input{apx_fast_counterexample}
\input{apx_slow_counterexample}

%% file: apx_first.tex
\subsection{Counterexample Outside Range}
\label{apx:outside_range}

\begin{figure}[h]
    \centering
    \includegraphics[width=\columnwidth]{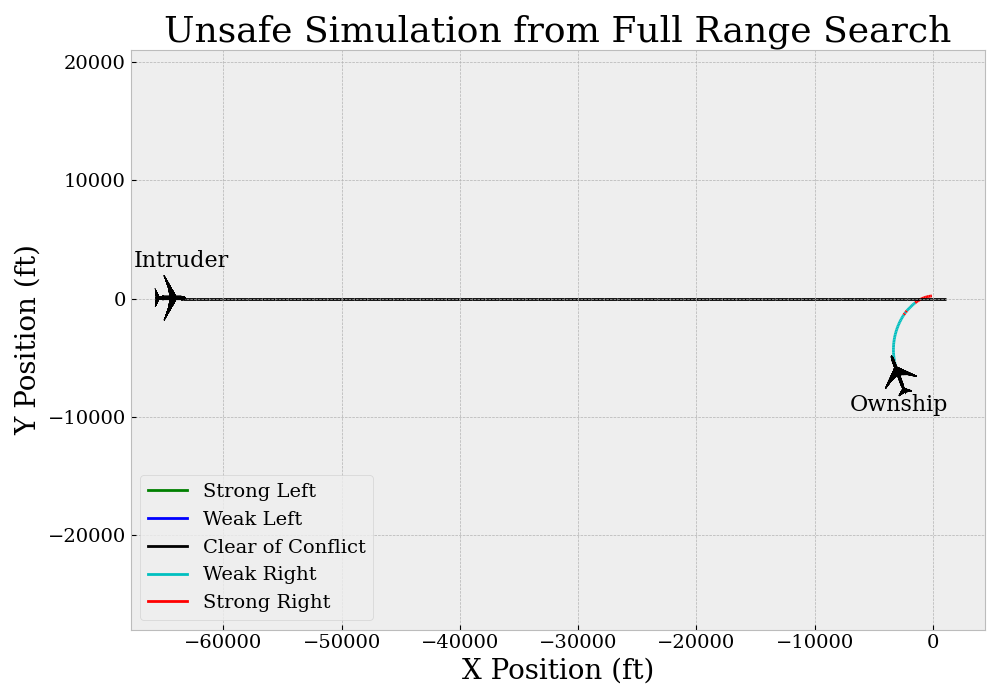}
    \caption{The network advises the ownship to turn immediately once the intruder is in range, when $\rho$ becomes less than 60760 ft, but a collision still occurs.}
    \label{fig:first_counterexample}
\end{figure}

When doing a full proof of safety of the system, counterexamples were generated for the original, nonquantized system.
In this scenario, the planes were at the same altitude, $\dot{\tau} = 0$, and the system commanded a turn as soon as the intruder was in range.
Nonetheless a collision occurred.
This indicates perhaps a larger range for $\rho$ should be considered, which would require retraining the network with more data in this range.
A visualization of the counterexample is shown in Figure~\ref{fig:first_counterexample}.
A video of the simulation is online: \url{https://youtu.be/WZWPKwIXfcM}.
The full trace is provided in Table~\ref{tab:first_unsafe}.
Row 3 of the table show the system issuing a turn command as soon as the distance between aircraft $\rho$ becomes less than 60760 ft.

\begin{center}
{\setlength{\tabcolsep}{5pt}
\scriptsize
\begin{longtable}{@{}lllllrrr@{}}
\caption{Unsafe case with immediate command. The initial state is $\rho$ = 62001.19897399513 ft, $\theta$ = 1.105638365566048 rad, $\psi=-1.9313853026445638$ rad, $v_{own}$ = 140.4154485909307 ft/sec, and $v_{int}$ = 1113.19526 ft/sec.} \label{tab:first_unsafe} \\
%%%%%%%%%%%%%%%%%%%%%%%%%%%
% Auto-generated
% The unrounded initial state is $\rho$ = 62001.19897399513 ft, $\theta$ = 1.105638365566048 rad, $\psi=-1.9313853026445638$ rad, $v_{own}$ = 140.4154485909307 ft/sec, and $v_{int}$ = 1113.19526 ft/sec.
\toprule
Step & $\alpha_\text{prev}$ & Cmd & $\rho$ (ft) & $\theta$ (deg) & $\psi$ (deg) \\
\midrule
1 & \textsc{coc} & \textsc{coc} & 62001.2 & 63.35 & -110.66 \\
2 & \textsc{coc} & \textsc{coc} & 60831.1 & 63.36 & -110.66 \\
\rowcolor{Red}
3 & \textsc{coc} & \textsc{wr} & 59661.0 & 63.37 & -110.66 \\
4 & \textsc{wr} & \textsc{wr} & 58492.6 & 64.88 & -109.16 \\
5 & \textsc{wr} & \textsc{wr} & 57327.4 & 66.39 & -107.66 \\
6 & \textsc{wr} & \textsc{wr} & 56165.7 & 67.90 & -106.16 \\
7 & \textsc{wr} & \textsc{wr} & 55007.4 & 69.42 & -104.66 \\
8 & \textsc{wr} & \textsc{wr} & 53852.5 & 70.94 & -103.16 \\
9 & \textsc{wr} & \textsc{wr} & 52701.1 & 72.46 & -101.66 \\
10 & \textsc{wr} & \textsc{wr} & 51553.2 & 73.98 & -100.16 \\
11 & \textsc{wr} & \textsc{wr} & 50408.8 & 75.51 & -98.66 \\
12 & \textsc{wr} & \textsc{wr} & 49268.0 & 77.03 & -97.16 \\
13 & \textsc{wr} & \textsc{wr} & 48130.8 & 78.56 & -95.66 \\
14 & \textsc{wr} & \textsc{wr} & 46997.3 & 80.09 & -94.16 \\
15 & \textsc{wr} & \textsc{wr} & 45867.3 & 81.63 & -92.66 \\
16 & \textsc{wr} & \textsc{wr} & 44741.0 & 83.16 & -91.16 \\
17 & \textsc{wr} & \textsc{wr} & 43618.4 & 84.70 & -89.66 \\
18 & \textsc{wr} & \textsc{wr} & 42499.5 & 86.24 & -88.16 \\
19 & \textsc{wr} & \textsc{wr} & 41384.3 & 87.79 & -86.66 \\
20 & \textsc{wr} & \textsc{wr} & 40272.7 & 89.33 & -85.16 \\
21 & \textsc{wr} & \textsc{wr} & 39164.9 & 90.88 & -83.66 \\
22 & \textsc{wr} & \textsc{wr} & 38060.7 & 92.44 & -82.16 \\
23 & \textsc{wr} & \textsc{wr} & 36960.3 & 93.99 & -80.66 \\
24 & \textsc{wr} & \textsc{wr} & 35863.6 & 95.55 & -79.16 \\
25 & \textsc{wr} & \textsc{wr} & 34770.6 & 97.11 & -77.66 \\
26 & \textsc{wr} & \textsc{wr} & 33681.2 & 98.67 & -76.16 \\
27 & \textsc{wr} & \textsc{wr} & 32595.6 & 100.24 & -74.66 \\
28 & \textsc{wr} & \textsc{wr} & 31513.6 & 101.81 & -73.16 \\
29 & \textsc{wr} & \textsc{wr} & 30435.2 & 103.38 & -71.66 \\
30 & \textsc{wr} & \textsc{wr} & 29360.5 & 104.96 & -70.16 \\
31 & \textsc{wr} & \textsc{wr} & 28289.4 & 106.54 & -68.66 \\
32 & \textsc{wr} & \textsc{wr} & 27221.9 & 108.13 & -67.16 \\
33 & \textsc{wr} & \textsc{wr} & 26157.9 & 109.72 & -65.66 \\
34 & \textsc{wr} & \textsc{wr} & 25097.5 & 111.32 & -64.16 \\
35 & \textsc{wr} & \textsc{wr} & 24040.5 & 112.92 & -62.66 \\
36 & \textsc{wr} & \textsc{wr} & 22987.0 & 114.53 & -61.16 \\
37 & \textsc{wr} & \textsc{wr} & 21937.0 & 116.14 & -59.66 \\
38 & \textsc{wr} & \textsc{wr} & 20890.3 & 117.76 & -58.16 \\
39 & \textsc{wr} & \textsc{sr} & 19847.0 & 119.39 & -56.66 \\
40 & \textsc{sr} & \textsc{wr} & 18808.6 & 122.52 & -53.66 \\
41 & \textsc{wr} & \textsc{sr} & 17775.0 & 124.16 & -52.16 \\
42 & \textsc{sr} & \textsc{wr} & 16746.0 & 127.30 & -49.16 \\
43 & \textsc{wr} & \textsc{wr} & 15721.5 & 128.96 & -47.66 \\
44 & \textsc{wr} & \textsc{wr} & 14700.0 & 130.62 & -46.16 \\
45 & \textsc{wr} & \textsc{wr} & 13681.4 & 132.30 & -44.66 \\
46 & \textsc{wr} & \textsc{wr} & 12665.7 & 134.00 & -43.16 \\
47 & \textsc{wr} & \textsc{wr} & 11652.8 & 135.72 & -41.66 \\
48 & \textsc{wr} & \textsc{wr} & 10642.7 & 137.46 & -40.16 \\
49 & \textsc{wr} & \textsc{sr} & 9635.3 & 139.25 & -38.66 \\
50 & \textsc{sr} & \textsc{sr} & 8631.7 & 142.57 & -35.66 \\
51 & \textsc{sr} & \textsc{sr} & 7632.8 & 145.92 & -32.66 \\
52 & \textsc{sr} & \textsc{sr} & 6638.5 & 149.34 & -29.66 \\
53 & \textsc{sr} & \textsc{sr} & 5648.3 & 152.84 & -26.66 \\
54 & \textsc{sr} & \textsc{sr} & 4661.9 & 156.46 & -23.66 \\
55 & \textsc{sr} & \textsc{sr} & 3679.3 & 160.32 & -20.66 \\
56 & \textsc{sr} & \textsc{sr} & 2700.4 & 164.66 & -17.66 \\
57 & \textsc{sr} & \textsc{sr} & 1726.2 & 170.27 & -14.66 \\
58 & \textsc{sr} & \textsc{sr} & 764.9 & -178.03 & -11.66 \\
59 & \textsc{sr} & \textsc{sr} & 309.3 & -50.16 & -8.66 \\
\bottomrule
%%%%%%%%%%%%%%%%%%%%%%%%%%%
\end{longtable}
}
\end{center}

%% file: apx_causecrash.tex
\subsection{Counterexample with Safety System Causing Collision}
\label{apx:causecrash}

\begin{figure}[h]
    \centering
    \includegraphics[width=\columnwidth]{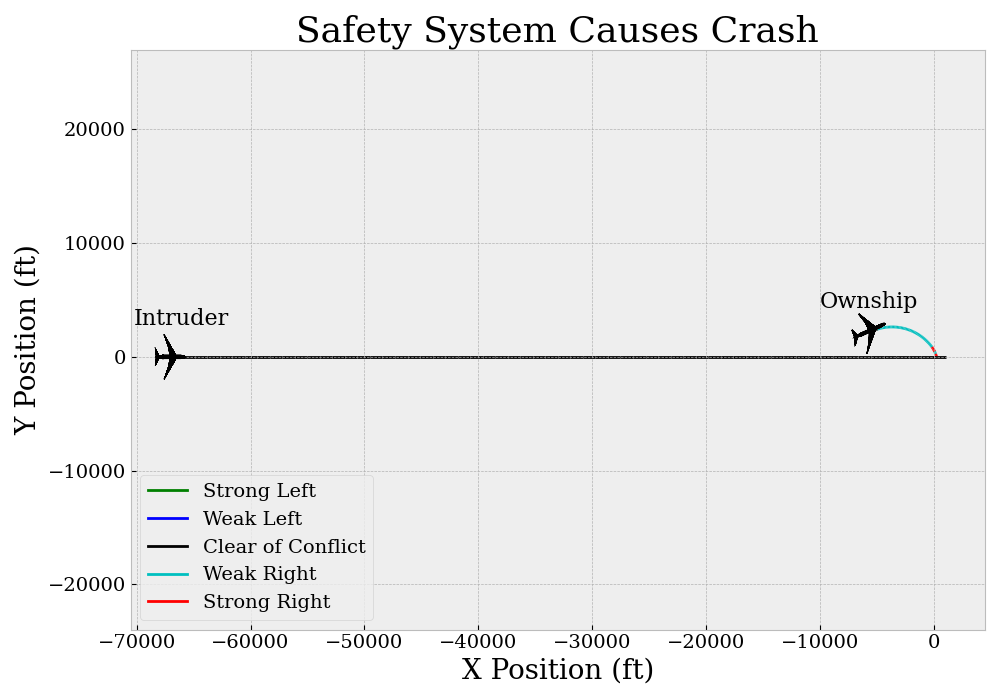}
    \caption{The safety system advises the ownship to turn into the path of the intruder, causing a collision that would have otherwise not occurred.}
    \label{fig:causecrash_counterexample}
\end{figure}

In this scenario, the planes are at the same altitude, $\dot{\tau} = 0$, and the ownship has cleared the path of the intruder.
Nonetheless, the system advises the ownship to turn, which eventually leads it back into the path of the intruder, causing a near mid-air collision.
A visualization of the unsafe scenario is shown in Figure~\ref{fig:causecrash_counterexample}.
A video of the simulation is online: \url{https://youtu.be/dDwRiv_Kh2M}.
The full trace is provided in Table~\ref{tab:causecrash}.
The unrounded initial state is $\rho$ = 61462.16874158125 ft, $\theta$ = 2.8797448888478536 rad, $\psi=-0.2973898012094359$ rad, with velocities $v_{own}$ = 114.27575493691512 ft/sec, and $v_{int}$ = 1100.31313 ft/sec.

\clearpage
\begin{center}
{\setlength{\tabcolsep}{5pt}
\scriptsize
\begin{longtable}{@{}lllllrrr@{}}
\caption{Unsafe case where the safety system causes a crash.} \label{tab:causecrash} \\
%%%%%%%%%%%%%%%%%%%%%%%%%%%
% Auto-generated
% The unrounded initial state is $\rho$ = 61462.16874158125 ft, $\theta$ = 2.8797448888478536 rad, $\psi=-0.2973898012094359$ rad, $v_{own}$ = 114.27575493691512 ft/sec, and $v_{int}$ = 1100.31313 ft/sec.
\toprule
Step & $\alpha_\text{prev}$ & Cmd & $\rho$ (ft) & $\theta$ (deg) & $\psi$ (deg) \\
\midrule
1 & \textsc{coc} & \textsc{coc} & 61462.2 & 165.00 & -17.04 \\
2 & \textsc{coc} & \textsc{coc} & 60473.0 & 165.06 & -17.04 \\
3 & \textsc{coc} & \textsc{coc} & 59483.9 & 165.13 & -17.04 \\
4 & \textsc{coc} & \textsc{coc} & 58494.8 & 165.20 & -17.04 \\
5 & \textsc{coc} & \textsc{coc} & 57505.9 & 165.27 & -17.04 \\
6 & \textsc{coc} & \textsc{coc} & 56517.0 & 165.35 & -17.04 \\
7 & \textsc{coc} & \textsc{coc} & 55528.3 & 165.42 & -17.04 \\
8 & \textsc{coc} & \textsc{wr} & 54539.6 & 165.50 & -17.04 \\
9 & \textsc{wr} & \textsc{wr} & 53551.4 & 167.08 & -15.54 \\
10 & \textsc{wr} & \textsc{wr} & 52564.0 & 168.66 & -14.04 \\
11 & \textsc{wr} & \textsc{wr} & 51577.3 & 170.25 & -12.54 \\
12 & \textsc{wr} & \textsc{wr} & 50591.2 & 171.83 & -11.04 \\
13 & \textsc{wr} & \textsc{wr} & 49605.7 & 173.41 & -9.54 \\
14 & \textsc{wr} & \textsc{wr} & 48620.5 & 174.99 & -8.04 \\
15 & \textsc{wr} & \textsc{wr} & 47635.8 & 176.57 & -6.54 \\
16 & \textsc{wr} & \textsc{wr} & 46651.3 & 178.15 & -5.04 \\
17 & \textsc{wr} & \textsc{wr} & 45666.9 & 179.73 & -3.54 \\
18 & \textsc{wr} & \textsc{wr} & 44682.7 & -178.69 & -2.04 \\
19 & \textsc{wr} & \textsc{wr} & 43698.5 & -177.12 & -0.54 \\
20 & \textsc{wr} & \textsc{wr} & 42714.3 & -175.54 & 0.96 \\
21 & \textsc{wr} & \textsc{wr} & 41729.8 & -173.96 & 2.46 \\
22 & \textsc{wr} & \textsc{wr} & 40745.2 & -172.38 & 3.96 \\
23 & \textsc{wr} & \textsc{wr} & 39760.2 & -170.80 & 5.46 \\
24 & \textsc{wr} & \textsc{wr} & 38774.8 & -169.23 & 6.96 \\
25 & \textsc{wr} & \textsc{wr} & 37788.9 & -167.65 & 8.46 \\
26 & \textsc{wr} & \textsc{wr} & 36802.5 & -166.08 & 9.96 \\
27 & \textsc{wr} & \textsc{wr} & 35815.4 & -164.50 & 11.46 \\
28 & \textsc{wr} & \textsc{wr} & 34827.5 & -162.92 & 12.96 \\
29 & \textsc{wr} & \textsc{wr} & 33838.9 & -161.35 & 14.46 \\
30 & \textsc{wr} & \textsc{wr} & 32849.3 & -159.78 & 15.96 \\
31 & \textsc{wr} & \textsc{wr} & 31858.8 & -158.20 & 17.46 \\
32 & \textsc{wr} & \textsc{wr} & 30867.2 & -156.63 & 18.96 \\
33 & \textsc{wr} & \textsc{wr} & 29874.4 & -155.06 & 20.46 \\
34 & \textsc{wr} & \textsc{wr} & 28880.5 & -153.48 & 21.96 \\
35 & \textsc{wr} & \textsc{wr} & 27885.2 & -151.91 & 23.46 \\
36 & \textsc{wr} & \textsc{wr} & 26888.6 & -150.34 & 24.96 \\
37 & \textsc{wr} & \textsc{wr} & 25890.6 & -148.77 & 26.46 \\
38 & \textsc{wr} & \textsc{wr} & 24891.0 & -147.20 & 27.96 \\
39 & \textsc{wr} & \textsc{wr} & 23889.9 & -145.63 & 29.46 \\
40 & \textsc{wr} & \textsc{wr} & 22887.1 & -144.06 & 30.96 \\
41 & \textsc{wr} & \textsc{wr} & 21882.6 & -142.48 & 32.46 \\
42 & \textsc{wr} & \textsc{wr} & 20876.3 & -140.91 & 33.96 \\
43 & \textsc{wr} & \textsc{wr} & 19868.2 & -139.34 & 35.46 \\
44 & \textsc{wr} & \textsc{wr} & 18858.1 & -137.77 & 36.96 \\
45 & \textsc{wr} & \textsc{wr} & 17846.0 & -136.19 & 38.46 \\
46 & \textsc{wr} & \textsc{wr} & 16832.0 & -134.62 & 39.96 \\
47 & \textsc{wr} & \textsc{wr} & 15815.8 & -133.04 & 41.46 \\
48 & \textsc{wr} & \textsc{wr} & 14797.4 & -131.46 & 42.96 \\
49 & \textsc{wr} & \textsc{wr} & 13776.9 & -129.87 & 44.46 \\
50 & \textsc{wr} & \textsc{wr} & 12754.1 & -128.28 & 45.96 \\
51 & \textsc{wr} & \textsc{wr} & 11728.9 & -126.69 & 47.46 \\
52 & \textsc{wr} & \textsc{wr} & 10701.4 & -125.08 & 48.96 \\
53 & \textsc{wr} & \textsc{wr} & 9671.5 & -123.46 & 50.46 \\
54 & \textsc{wr} & \textsc{wr} & 8639.2 & -121.83 & 51.96 \\
55 & \textsc{wr} & \textsc{sr} & 7604.4 & -120.17 & 53.46 \\
56 & \textsc{sr} & \textsc{sr} & 6565.7 & -116.98 & 56.46 \\
57 & \textsc{sr} & \textsc{sr} & 5521.8 & -113.74 & 59.46 \\
58 & \textsc{sr} & \textsc{sr} & 4472.5 & -110.43 & 62.46 \\
59 & \textsc{sr} & \textsc{wr} & 3417.8 & -106.96 & 65.46 \\
60 & \textsc{wr} & \textsc{wr} & 2359.3 & -104.60 & 66.96 \\
61 & \textsc{wr} & \textsc{sr} & 1299.3 & -100.87 & 68.46 \\
62 & \textsc{sr} & \textsc{sr} & 253.5 & -76.83 & 71.46 \\
\bottomrule
%%%%%%%%%%%%%%%%%%%%%%%%%%%
\end{longtable}
}
\end{center}

%% file: apx_fast_counterexample.tex
\subsection{Counterexample with Fast Ownship}
\label{apx:fast_ownship}

\begin{figure}[h]
    \centering
    \includegraphics[width=\columnwidth]{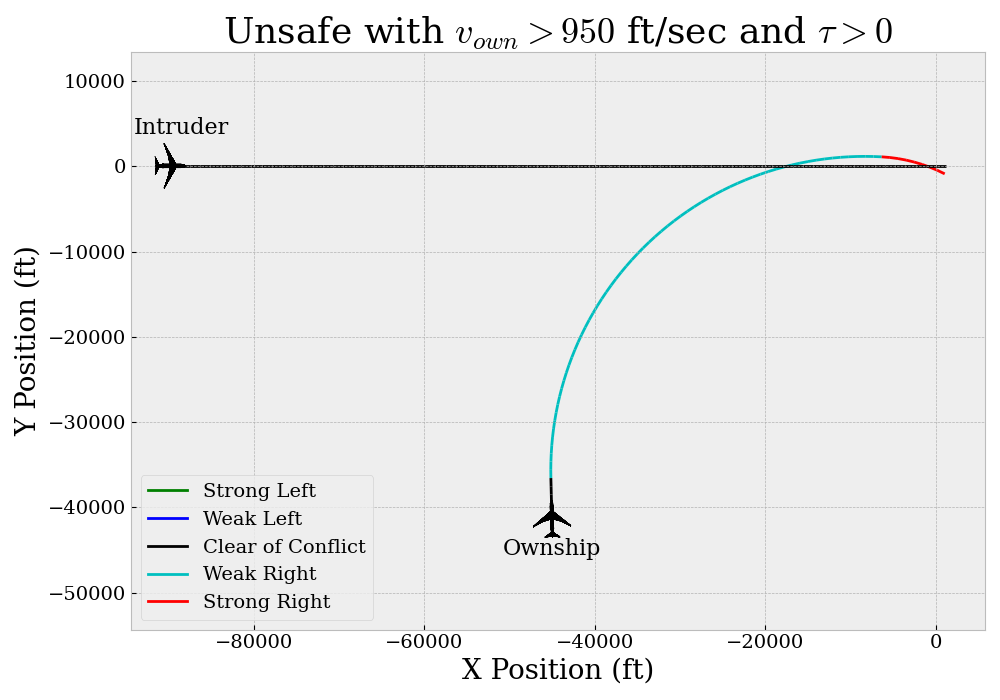}
    \caption{Unsafe case with fast ownship ($v_{own} = 964.1$ ft/sec) and $\tau_\text{init} = 75$ sec.}
    \label{fig:fast_counterexample}
\end{figure}

Using the described quantized backreachability approach, we found an unsafe situation with $v_{own} > 950$ ft/sec.
In this case, the system was safe for in-plane flight, when $\dot{\tau} = 0$, but unsafe when there was a difference in vertical velocity, $\dot{\tau} = -1$.
A visualization of the counterexample is shown in Figure~\ref{fig:fast_counterexample}.
A video of the simulation is online: \url{https://youtu.be/F_bykLR9lJw}.
From the visualization, it looks like the collision was avoidable if the system did not continue to advise a right turn for the final few seconds.
The full trace is provided in Table~\ref{tab:fast_own_unsafe}.

\clearpage
\begin{center}
{\setlength{\tabcolsep}{5pt}
\scriptsize
\begin{longtable}{@{}lllllrrr@{}}
\caption{Unsafe case with fast ownship with $v_{own} = 964.1$. The full initial state is $\rho$ = 61019.45806978694 ft, $\theta$ = 0.8007909138337812 rad, $\psi=-1.5953555128455696$ rad, $v_{own}$ = 964.0586611224201 ft/sec, and $v_{int}$ = 1198.4375 ft/sec.} \label{tab:fast_own_unsafe} \\
%%%%%%%%%%%%%%%%%%%%%%%%%%%
% Auto-generated
% The unrounded initial state is $\rho$ = 61019.45806978694 ft, $\theta$ = 0.8007909138337812 rad, $\psi=-1.5953555128455696$ rad, $v_{own}$ = 964.0586611224201 ft/sec, and $v_{int}$ = 1198.4375 ft/sec.
\toprule
Step & $\alpha_\text{prev}$ & $\tau$ & Net & Cmd & $\rho$ (ft) & $\theta$ (deg) & $\psi$ (deg) \\
\midrule
1 & \textsc{coc} & 75 & $N_{1,8}$ & \textsc{coc} & 61019.5 & 45.88 & -91.41 \\
2 & \textsc{coc} & 74 & $N_{1,8}$ & \textsc{coc} & 59467.9 & 45.77 & -91.41 \\
3 & \textsc{coc} & 73 & $N_{1,8}$ & \textsc{coc} & 57916.5 & 45.64 & -91.41 \\
4 & \textsc{coc} & 72 & $N_{1,8}$ & \textsc{coc} & 56365.5 & 45.51 & -91.41 \\
5 & \textsc{coc} & 71 & $N_{1,8}$ & \textsc{coc} & 54814.7 & 45.38 & -91.41 \\
6 & \textsc{coc} & 70 & $N_{1,7}$ & \textsc{wr} & 53264.3 & 45.23 & -91.41 \\
7 & \textsc{wr} & 69 & $N_{3,7}$ & \textsc{wr} & 51723.3 & 46.59 & -89.91 \\
8 & \textsc{wr} & 68 & $N_{3,7}$ & \textsc{wr} & 50200.9 & 47.96 & -88.41 \\
9 & \textsc{wr} & 67 & $N_{3,7}$ & \textsc{wr} & 48697.4 & 49.34 & -86.91 \\
10 & \textsc{wr} & 66 & $N_{3,7}$ & \textsc{wr} & 47213.4 & 50.73 & -85.41 \\
11 & \textsc{wr} & 65 & $N_{3,7}$ & \textsc{wr} & 45749.0 & 52.13 & -83.91 \\
12 & \textsc{wr} & 64 & $N_{3,7}$ & \textsc{wr} & 44304.6 & 53.55 & -82.41 \\
13 & \textsc{wr} & 63 & $N_{3,7}$ & \textsc{wr} & 42880.6 & 54.98 & -80.91 \\
14 & \textsc{wr} & 62 & $N_{3,7}$ & \textsc{wr} & 41477.4 & 56.43 & -79.41 \\
15 & \textsc{wr} & 61 & $N_{3,7}$ & \textsc{wr} & 40095.1 & 57.90 & -77.91 \\
16 & \textsc{wr} & 60 & $N_{3,7}$ & \textsc{wr} & 38734.3 & 59.38 & -76.41 \\
17 & \textsc{wr} & 59 & $N_{3,7}$ & \textsc{wr} & 37395.2 & 60.88 & -74.91 \\
18 & \textsc{wr} & 58 & $N_{3,7}$ & \textsc{wr} & 36078.2 & 62.40 & -73.41 \\
19 & \textsc{wr} & 57 & $N_{3,7}$ & \textsc{wr} & 34783.5 & 63.94 & -71.91 \\
20 & \textsc{wr} & 56 & $N_{3,7}$ & \textsc{wr} & 33511.5 & 65.50 & -70.41 \\
21 & \textsc{wr} & 55 & $N_{3,6}$ & \textsc{wr} & 32262.5 & 67.09 & -68.91 \\
22 & \textsc{wr} & 54 & $N_{3,6}$ & \textsc{wr} & 31036.9 & 68.70 & -67.41 \\
23 & \textsc{wr} & 53 & $N_{3,6}$ & \textsc{wr} & 29835.0 & 70.34 & -65.91 \\
24 & \textsc{wr} & 52 & $N_{3,6}$ & \textsc{wr} & 28657.0 & 72.00 & -64.41 \\
25 & \textsc{wr} & 51 & $N_{3,6}$ & \textsc{wr} & 27503.3 & 73.70 & -62.91 \\
26 & \textsc{wr} & 50 & $N_{3,6}$ & \textsc{wr} & 26374.2 & 75.43 & -61.41 \\
27 & \textsc{wr} & 49 & $N_{3,6}$ & \textsc{wr} & 25270.0 & 77.19 & -59.91 \\
28 & \textsc{wr} & 48 & $N_{3,6}$ & \textsc{wr} & 24191.1 & 78.99 & -58.41 \\
29 & \textsc{wr} & 47 & $N_{3,6}$ & \textsc{wr} & 23137.7 & 80.83 & -56.91 \\
30 & \textsc{wr} & 46 & $N_{3,6}$ & \textsc{wr} & 22110.1 & 82.71 & -55.41 \\
31 & \textsc{wr} & 45 & $N_{3,6}$ & \textsc{wr} & 21108.6 & 84.64 & -53.91 \\
32 & \textsc{wr} & 44 & $N_{3,6}$ & \textsc{wr} & 20133.6 & 86.62 & -52.41 \\
33 & \textsc{wr} & 43 & $N_{3,6}$ & \textsc{wr} & 19185.4 & 88.65 & -50.91 \\
34 & \textsc{wr} & 42 & $N_{3,6}$ & \textsc{wr} & 18264.1 & 90.73 & -49.41 \\
35 & \textsc{wr} & 41 & $N_{3,6}$ & \textsc{wr} & 17370.3 & 92.88 & -47.91 \\
36 & \textsc{wr} & 40 & $N_{3,6}$ & \textsc{wr} & 16504.0 & 95.09 & -46.41 \\
37 & \textsc{wr} & 39 & $N_{3,6}$ & \textsc{wr} & 15665.7 & 97.37 & -44.91 \\
38 & \textsc{wr} & 38 & $N_{3,6}$ & \textsc{wr} & 14855.5 & 99.72 & -43.41 \\
39 & \textsc{wr} & 37 & $N_{3,6}$ & \textsc{wr} & 14073.9 & 102.15 & -41.91 \\
40 & \textsc{wr} & 36 & $N_{3,6}$ & \textsc{wr} & 13320.9 & 104.67 & -40.41 \\
41 & \textsc{wr} & 35 & $N_{3,5}$ & \textsc{wr} & 12597.0 & 107.27 & -38.91 \\
42 & \textsc{wr} & 34 & $N_{3,5}$ & \textsc{wr} & 11902.2 & 109.98 & -37.41 \\
43 & \textsc{wr} & 33 & $N_{3,5}$ & \textsc{wr} & 11236.9 & 112.78 & -35.91 \\
44 & \textsc{wr} & 32 & $N_{3,5}$ & \textsc{wr} & 10601.1 & 115.69 & -34.41 \\
45 & \textsc{wr} & 31 & $N_{3,5}$ & \textsc{wr} & 9995.0 & 118.72 & -32.91 \\
46 & \textsc{wr} & 30 & $N_{3,5}$ & \textsc{wr} & 9418.6 & 121.86 & -31.41 \\
47 & \textsc{wr} & 29 & $N_{3,5}$ & \textsc{wr} & 8872.0 & 125.12 & -29.91 \\
48 & \textsc{wr} & 28 & $N_{3,5}$ & \textsc{wr} & 8355.0 & 128.51 & -28.41 \\
49 & \textsc{wr} & 27 & $N_{3,5}$ & \textsc{wr} & 7867.4 & 132.02 & -26.91 \\
50 & \textsc{wr} & 26 & $N_{3,5}$ & \textsc{wr} & 7409.0 & 135.66 & -25.41 \\
51 & \textsc{wr} & 25 & $N_{3,5}$ & \textsc{wr} & 6979.1 & 139.43 & -23.91 \\
52 & \textsc{wr} & 24 & $N_{3,5}$ & \textsc{wr} & 6577.1 & 143.31 & -22.41 \\
53 & \textsc{wr} & 23 & $N_{3,5}$ & \textsc{wr} & 6202.1 & 147.31 & -20.91 \\
54 & \textsc{wr} & 22 & $N_{3,5}$ & \textsc{wr} & 5853.0 & 151.41 & -19.41 \\
55 & \textsc{wr} & 21 & $N_{3,5}$ & \textsc{wr} & 5528.4 & 155.59 & -17.91 \\
56 & \textsc{wr} & 20 & $N_{3,5}$ & \textsc{wr} & 5226.7 & 159.85 & -16.41 \\
57 & \textsc{wr} & 19 & $N_{3,5}$ & \textsc{wr} & 4946.0 & 164.15 & -14.91 \\
58 & \textsc{wr} & 18 & $N_{3,5}$ & \textsc{wr} & 4684.2 & 168.48 & -13.41 \\
59 & \textsc{wr} & 17 & $N_{3,5}$ & \textsc{wr} & 4438.9 & 172.82 & -11.91 \\
60 & \textsc{wr} & 16 & $N_{3,5}$ & \textsc{wr} & 4207.6 & 177.13 & -10.41 \\
61 & \textsc{wr} & 15 & $N_{3,4}$ & \textsc{wr} & 3987.6 & -178.59 & -8.91 \\
62 & \textsc{wr} & 14 & $N_{3,4}$ & \textsc{wr} & 3776.1 & -174.39 & -7.41 \\
63 & \textsc{wr} & 13 & $N_{3,4}$ & \textsc{wr} & 3570.3 & -170.27 & -5.91 \\
64 & \textsc{wr} & 12 & $N_{3,4}$ & \textsc{wr} & 3367.4 & -166.25 & -4.41 \\
65 & \textsc{wr} & 11 & $N_{3,4}$ & \textsc{wr} & 3164.7 & -162.36 & -2.91 \\
66 & \textsc{wr} & 10 & $N_{3,4}$ & \textsc{wr} & 2959.6 & -158.61 & -1.41 \\
67 & \textsc{wr} & 9 & $N_{3,4}$ & \textsc{wr} & 2749.6 & -155.00 & 0.09 \\
68 & \textsc{wr} & 8 & $N_{3,4}$ & \textsc{wr} & 2532.4 & -151.56 & 1.59 \\
69 & \textsc{wr} & 7 & $N_{3,3}$ & \textsc{sr} & 2305.8 & -148.29 & 3.09 \\
70 & \textsc{sr} & 6 & $N_{5,3}$ & \textsc{sr} & 2060.8 & -144.01 & 6.09 \\
71 & \textsc{sr} & 5 & $N_{5,3}$ & \textsc{sr} & 1786.7 & -140.67 & 9.09 \\
72 & \textsc{sr} & 4 & $N_{5,3}$ & \textsc{sr} & 1480.9 & -138.70 & 12.09 \\
73 & \textsc{sr} & 3 & $N_{5,2}$ & \textsc{sr} & 1144.3 & -139.22 & 15.09 \\
74 & \textsc{sr} & 2 & $N_{5,2}$ & \textsc{sr} & 788.1 & -145.64 & 18.09 \\
75 & \textsc{sr} & 1 & $N_{5,2}$ & \textsc{sr} & 477.4 & -171.30 & 21.09 \\
76 & \textsc{sr} & 0 & $N_{5,1}$ & \textsc{sr} & 498.5 & 132.55 & 24.09 \\
\bottomrule

%%%%%%%%%%%%%%%%%%%%%%%%%%%
\end{longtable}
}
\end{center}

%% file: apx_slow_counterexample.tex
\subsection{Counterexample with $v_{int} < 400$ ft/sec}
\label{apx:slow_intruder}

\begin{figure}[h]
    \centering
    \includegraphics[width=\columnwidth]{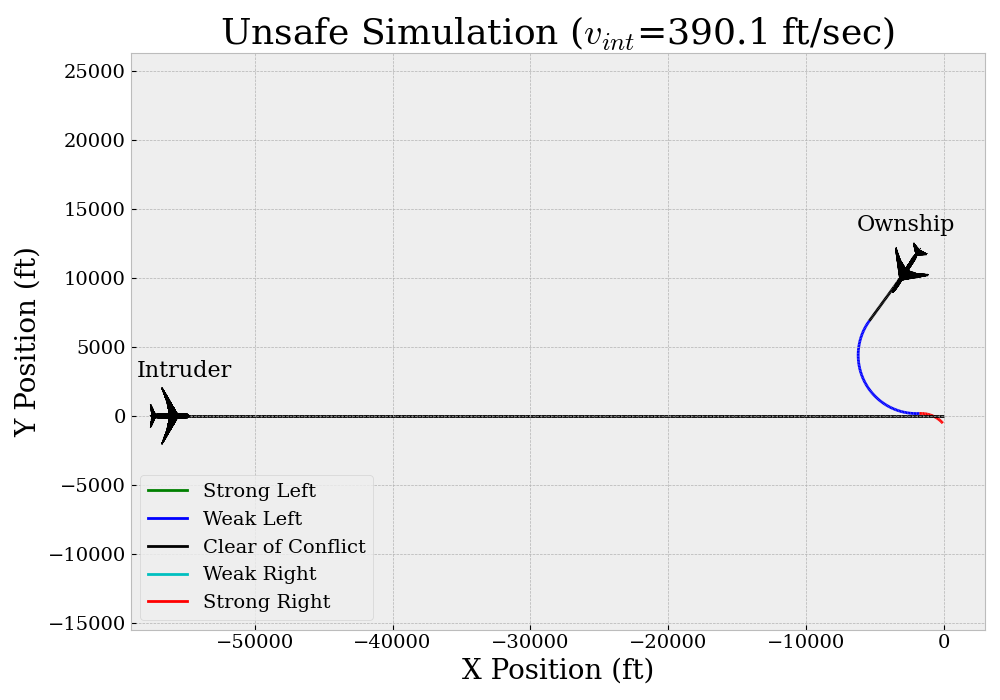}
    \caption{Unsafe case with slow intruder. Aircraft are shown at the initial positions.}
    \label{fig:slow_counterexample}
\end{figure}

Using the described quantized backreachability approach, we found an unsafe situation with $v_{int} < 400$ ft/sec.
A video of the simulation is online, \url{https://youtu.be/7B_-k0qpZTo}, and a visualization of the counterexample is shown in Figure~\ref{fig:slow_counterexample}.

The 159 second trace is provided in Table~\ref{tab:slow_int_unsafe}.
Examining the trace, an interesting observation is that the command switch from weak-left to strong right at step 142 corresponds to the relative position angle $\theta$ wrapping from $-\pi$ to $\pi$.
This discontinuity in the network input between successive steps is likely the cause of the eventual near mid-air collision.

\begin{center}
{\setlength{\tabcolsep}{5pt}
\scriptsize
\begin{longtable}{@{}lllrrr@{}}
\caption{Unsafe case with slow intruder with velocity $v_{int} = 390.1$ ft/sec and in-plane flight ($\tau = 0$ and $\dot{\tau} = 0$). The unrounded initial state is $\rho$ = 60959.597800102 ft, $\theta$ = -0.7461997148243538 rad, $\psi=2.1997877266124295$ rad, $v_{own}$ = 110.84814862335269 ft/sec, and $v_{int}$ = 390.10329256 ft/sec.}\label{tab:slow_int_unsafe} \\
%%%%%%%%%%%%%%%%%%%%%%%%%%%%%%5
% Auto-generated
% The unrounded initial state is $\rho$ = 60959.597800102 ft, $\theta$ = -0.7461997148243538 rad, $\psi=2.1997877266124295$ rad, $v_{own}$ = 110.84814862335269 ft/sec, and $v_{int}$ = 390.10329256 ft/sec.
\toprule
Step & $\alpha_\text{prev}$ & Cmd & $\rho$ (ft) & $\theta$ (deg) & $\psi$ (deg) \\
\midrule
1 & \textsc{coc} & \textsc{coc} & 60959.6 & -42.75 & 126.04 \\
2 & \textsc{coc} & \textsc{coc} & 60495.5 & -42.75 & 126.04 \\
3 & \textsc{coc} & \textsc{coc} & 60031.5 & -42.75 & 126.04 \\
4 & \textsc{coc} & \textsc{coc} & 59567.4 & -42.75 & 126.04 \\
5 & \textsc{coc} & \textsc{coc} & 59103.4 & -42.75 & 126.04 \\
6 & \textsc{coc} & \textsc{coc} & 58639.3 & -42.75 & 126.04 \\
7 & \textsc{coc} & \textsc{coc} & 58175.3 & -42.75 & 126.04 \\
8 & \textsc{coc} & \textsc{coc} & 57711.2 & -42.75 & 126.04 \\
9 & \textsc{coc} & \textsc{coc} & 57247.1 & -42.75 & 126.04 \\
10 & \textsc{coc} & \textsc{coc} & 56783.1 & -42.75 & 126.04 \\
11 & \textsc{coc} & \textsc{coc} & 56319.0 & -42.75 & 126.04 \\
12 & \textsc{coc} & \textsc{coc} & 55855.0 & -42.75 & 126.04 \\
13 & \textsc{coc} & \textsc{coc} & 55390.9 & -42.75 & 126.04 \\
14 & \textsc{coc} & \textsc{coc} & 54926.9 & -42.75 & 126.04 \\
15 & \textsc{coc} & \textsc{coc} & 54462.8 & -42.75 & 126.04 \\
16 & \textsc{coc} & \textsc{coc} & 53998.7 & -42.75 & 126.04 \\
17 & \textsc{coc} & \textsc{coc} & 53534.7 & -42.74 & 126.04 \\
18 & \textsc{coc} & \textsc{coc} & 53070.6 & -42.74 & 126.04 \\
19 & \textsc{coc} & \textsc{coc} & 52606.6 & -42.74 & 126.04 \\
20 & \textsc{coc} & \textsc{coc} & 52142.5 & -42.74 & 126.04 \\
21 & \textsc{coc} & \textsc{coc} & 51678.5 & -42.74 & 126.04 \\
22 & \textsc{coc} & \textsc{coc} & 51214.4 & -42.74 & 126.04 \\
23 & \textsc{coc} & \textsc{coc} & 50750.3 & -42.74 & 126.04 \\
24 & \textsc{coc} & \textsc{coc} & 50286.3 & -42.74 & 126.04 \\
25 & \textsc{coc} & \textsc{coc} & 49822.2 & -42.74 & 126.04 \\
26 & \textsc{coc} & \textsc{coc} & 49358.2 & -42.74 & 126.04 \\
27 & \textsc{coc} & \textsc{coc} & 48894.1 & -42.74 & 126.04 \\
28 & \textsc{coc} & \textsc{coc} & 48430.1 & -42.74 & 126.04 \\
29 & \textsc{coc} & \textsc{coc} & 47966.0 & -42.73 & 126.04 \\
30 & \textsc{coc} & \textsc{coc} & 47501.9 & -42.73 & 126.04 \\
31 & \textsc{coc} & \textsc{coc} & 47037.9 & -42.73 & 126.04 \\
32 & \textsc{coc} & \textsc{coc} & 46573.8 & -42.73 & 126.04 \\
33 & \textsc{coc} & \textsc{coc} & 46109.8 & -42.73 & 126.04 \\
34 & \textsc{coc} & \textsc{coc} & 45645.7 & -42.73 & 126.04 \\
35 & \textsc{coc} & \textsc{coc} & 45181.7 & -42.73 & 126.04 \\
36 & \textsc{coc} & \textsc{coc} & 44717.6 & -42.73 & 126.04 \\
37 & \textsc{coc} & \textsc{coc} & 44253.5 & -42.73 & 126.04 \\
38 & \textsc{coc} & \textsc{coc} & 43789.5 & -42.73 & 126.04 \\
39 & \textsc{coc} & \textsc{coc} & 43325.4 & -42.73 & 126.04 \\
40 & \textsc{coc} & \textsc{coc} & 42861.4 & -42.72 & 126.04 \\
41 & \textsc{coc} & \textsc{coc} & 42397.3 & -42.72 & 126.04 \\
42 & \textsc{coc} & \textsc{coc} & 41933.3 & -42.72 & 126.04 \\
43 & \textsc{coc} & \textsc{coc} & 41469.2 & -42.72 & 126.04 \\
44 & \textsc{coc} & \textsc{coc} & 41005.2 & -42.72 & 126.04 \\
45 & \textsc{coc} & \textsc{coc} & 40541.1 & -42.72 & 126.04 \\
46 & \textsc{coc} & \textsc{coc} & 40077.0 & -42.72 & 126.04 \\
47 & \textsc{coc} & \textsc{coc} & 39613.0 & -42.72 & 126.04 \\
48 & \textsc{coc} & \textsc{coc} & 39148.9 & -42.71 & 126.04 \\
49 & \textsc{coc} & \textsc{coc} & 38684.9 & -42.71 & 126.04 \\
50 & \textsc{coc} & \textsc{coc} & 38220.8 & -42.71 & 126.04 \\
51 & \textsc{coc} & \textsc{coc} & 37756.8 & -42.71 & 126.04 \\
52 & \textsc{coc} & \textsc{coc} & 37292.7 & -42.71 & 126.04 \\
53 & \textsc{coc} & \textsc{coc} & 36828.6 & -42.71 & 126.04 \\
54 & \textsc{coc} & \textsc{coc} & 36364.6 & -42.71 & 126.04 \\
55 & \textsc{coc} & \textsc{coc} & 35900.5 & -42.70 & 126.04 \\
56 & \textsc{coc} & \textsc{wl} & 35436.5 & -42.70 & 126.04 \\
57 & \textsc{wl} & \textsc{wl} & 34973.4 & -44.20 & 124.54 \\
58 & \textsc{wl} & \textsc{wl} & 34512.4 & -45.71 & 123.04 \\
59 & \textsc{wl} & \textsc{wl} & 34053.4 & -47.21 & 121.54 \\
60 & \textsc{wl} & \textsc{wl} & 33596.6 & -48.72 & 120.04 \\
61 & \textsc{wl} & \textsc{wl} & 33141.9 & -50.24 & 118.54 \\
62 & \textsc{wl} & \textsc{wl} & 32689.5 & -51.76 & 117.04 \\
63 & \textsc{wl} & \textsc{wl} & 32239.4 & -53.28 & 115.54 \\
64 & \textsc{wl} & \textsc{wl} & 31791.6 & -54.80 & 114.04 \\
65 & \textsc{wl} & \textsc{wl} & 31346.2 & -56.33 & 112.54 \\
66 & \textsc{wl} & \textsc{wl} & 30903.2 & -57.87 & 111.04 \\
67 & \textsc{wl} & \textsc{wl} & 30462.6 & -59.40 & 109.54 \\
68 & \textsc{wl} & \textsc{wl} & 30024.6 & -60.94 & 108.04 \\
69 & \textsc{wl} & \textsc{wl} & 29589.1 & -62.49 & 106.54 \\
70 & \textsc{wl} & \textsc{wl} & 29156.3 & -64.04 & 105.04 \\
71 & \textsc{wl} & \textsc{wl} & 28726.0 & -65.59 & 103.54 \\
72 & \textsc{wl} & \textsc{wl} & 28298.5 & -67.15 & 102.04 \\
73 & \textsc{wl} & \textsc{wl} & 27873.6 & -68.71 & 100.54 \\
74 & \textsc{wl} & \textsc{wl} & 27451.5 & -70.27 & 99.04 \\
75 & \textsc{wl} & \textsc{wl} & 27032.1 & -71.84 & 97.54 \\
76 & \textsc{wl} & \textsc{wl} & 26615.5 & -73.41 & 96.04 \\
77 & \textsc{wl} & \textsc{wl} & 26201.8 & -74.99 & 94.54 \\
78 & \textsc{wl} & \textsc{wl} & 25790.9 & -76.57 & 93.04 \\
79 & \textsc{wl} & \textsc{wl} & 25382.9 & -78.16 & 91.54 \\
80 & \textsc{wl} & \textsc{wl} & 24977.8 & -79.75 & 90.04 \\
81 & \textsc{wl} & \textsc{wl} & 24575.6 & -81.34 & 88.54 \\
82 & \textsc{wl} & \textsc{wl} & 24176.4 & -82.94 & 87.04 \\
83 & \textsc{wl} & \textsc{wl} & 23780.1 & -84.54 & 85.54 \\
84 & \textsc{wl} & \textsc{wl} & 23386.7 & -86.14 & 84.04 \\
85 & \textsc{wl} & \textsc{wl} & 22996.4 & -87.75 & 82.54 \\
86 & \textsc{wl} & \textsc{wl} & 22609.0 & -89.37 & 81.04 \\
87 & \textsc{wl} & \textsc{wl} & 22224.6 & -90.99 & 79.54 \\
88 & \textsc{wl} & \textsc{wl} & 21843.3 & -92.61 & 78.04 \\
89 & \textsc{wl} & \textsc{wl} & 21464.9 & -94.24 & 76.54 \\
90 & \textsc{wl} & \textsc{wl} & 21089.5 & -95.87 & 75.04 \\
91 & \textsc{wl} & \textsc{wl} & 20717.1 & -97.50 & 73.54 \\
92 & \textsc{wl} & \textsc{wl} & 20347.8 & -99.14 & 72.04 \\
93 & \textsc{wl} & \textsc{wl} & 19981.4 & -100.78 & 70.54 \\
94 & \textsc{wl} & \textsc{wl} & 19618.0 & -102.42 & 69.04 \\
95 & \textsc{wl} & \textsc{wl} & 19257.5 & -104.07 & 67.54 \\
96 & \textsc{wl} & \textsc{wl} & 18900.0 & -105.72 & 66.04 \\
97 & \textsc{wl} & \textsc{wl} & 18545.4 & -107.38 & 64.54 \\
98 & \textsc{wl} & \textsc{wl} & 18193.8 & -109.04 & 63.04 \\
99 & \textsc{wl} & \textsc{wl} & 17845.0 & -110.70 & 61.54 \\
100 & \textsc{wl} & \textsc{wl} & 17499.1 & -112.37 & 60.04 \\
101 & \textsc{wl} & \textsc{wl} & 17156.0 & -114.04 & 58.54 \\
102 & \textsc{wl} & \textsc{wl} & 16815.7 & -115.71 & 57.04 \\
103 & \textsc{wl} & \textsc{wl} & 16478.2 & -117.38 & 55.54 \\
104 & \textsc{wl} & \textsc{wl} & 16143.4 & -119.06 & 54.04 \\
105 & \textsc{wl} & \textsc{wl} & 15811.3 & -120.74 & 52.54 \\
106 & \textsc{wl} & \textsc{wl} & 15481.8 & -122.42 & 51.04 \\
107 & \textsc{wl} & \textsc{wl} & 15155.0 & -124.10 & 49.54 \\
108 & \textsc{wl} & \textsc{wl} & 14830.7 & -125.78 & 48.04 \\
109 & \textsc{wl} & \textsc{wl} & 14508.9 & -127.47 & 46.54 \\
110 & \textsc{wl} & \textsc{wl} & 14189.6 & -129.16 & 45.04 \\
111 & \textsc{wl} & \textsc{wl} & 13872.7 & -130.84 & 43.54 \\
112 & \textsc{wl} & \textsc{wl} & 13558.1 & -132.53 & 42.04 \\
113 & \textsc{wl} & \textsc{wl} & 13245.8 & -134.22 & 40.54 \\
114 & \textsc{wl} & \textsc{wl} & 12935.8 & -135.91 & 39.04 \\
115 & \textsc{wl} & \textsc{wl} & 12627.9 & -137.60 & 37.54 \\
116 & \textsc{wl} & \textsc{wl} & 12322.1 & -139.29 & 36.04 \\
117 & \textsc{wl} & \textsc{wl} & 12018.3 & -140.98 & 34.54 \\
118 & \textsc{wl} & \textsc{wl} & 11716.4 & -142.67 & 33.04 \\
119 & \textsc{wl} & \textsc{wl} & 11416.5 & -144.35 & 31.54 \\
120 & \textsc{wl} & \textsc{wl} & 11118.4 & -146.03 & 30.04 \\
121 & \textsc{wl} & \textsc{wl} & 10822.0 & -147.71 & 28.54 \\
122 & \textsc{wl} & \textsc{wl} & 10527.2 & -149.39 & 27.04 \\
123 & \textsc{wl} & \textsc{wl} & 10234.1 & -151.06 & 25.54 \\
124 & \textsc{wl} & \textsc{wl} & 9942.4 & -152.73 & 24.04 \\
125 & \textsc{wl} & \textsc{wl} & 9652.1 & -154.39 & 22.54 \\
126 & \textsc{wl} & \textsc{wl} & 9363.2 & -156.05 & 21.04 \\
127 & \textsc{wl} & \textsc{wl} & 9075.5 & -157.70 & 19.54 \\
128 & \textsc{wl} & \textsc{wl} & 8789.0 & -159.34 & 18.04 \\
129 & \textsc{wl} & \textsc{wl} & 8503.5 & -160.98 & 16.54 \\
130 & \textsc{wl} & \textsc{wl} & 8219.1 & -162.60 & 15.04 \\
131 & \textsc{wl} & \textsc{wl} & 7935.5 & -164.22 & 13.54 \\
132 & \textsc{wl} & \textsc{wl} & 7652.8 & -165.82 & 12.04 \\
133 & \textsc{wl} & \textsc{wl} & 7370.8 & -167.40 & 10.54 \\
134 & \textsc{wl} & \textsc{wl} & 7089.4 & -168.98 & 9.04 \\
135 & \textsc{wl} & \textsc{wl} & 6808.6 & -170.53 & 7.54 \\
136 & \textsc{wl} & \textsc{wl} & 6528.3 & -172.06 & 6.04 \\
137 & \textsc{wl} & \textsc{wl} & 6248.4 & -173.57 & 4.54 \\
138 & \textsc{wl} & \textsc{wl} & 5968.8 & -175.05 & 3.04 \\
139 & \textsc{wl} & \textsc{wl} & 5689.4 & -176.50 & 1.54 \\
140 & \textsc{wl} & \textsc{wl} & 5410.3 & -177.92 & 0.04 \\
141 & \textsc{wl} & \textsc{wl} & 5131.3 & -179.29 & -1.46 \\
\rowcolor{Red}
142 & \textsc{wl} & \textsc{sr} & 4852.3 & 179.39 & -2.96 \\
143 & \textsc{sr} & \textsc{sr} & 4573.4 & -177.43 & 0.04 \\
144 & \textsc{sr} & \textsc{sr} & 4294.2 & -174.31 & 3.04 \\
145 & \textsc{sr} & \textsc{sr} & 4014.5 & -171.25 & 6.04 \\
146 & \textsc{sr} & \textsc{sr} & 3733.9 & -168.27 & 9.04 \\
147 & \textsc{sr} & \textsc{sr} & 3452.1 & -165.39 & 12.04 \\
148 & \textsc{sr} & \textsc{sr} & 3168.9 & -162.63 & 15.04 \\
149 & \textsc{sr} & \textsc{sr} & 2884.1 & -160.02 & 18.04 \\
150 & \textsc{sr} & \textsc{sr} & 2597.5 & -157.63 & 21.04 \\
151 & \textsc{sr} & \textsc{sr} & 2309.1 & -155.52 & 24.04 \\
152 & \textsc{sr} & \textsc{sr} & 2019.2 & -153.81 & 27.04 \\
153 & \textsc{sr} & \textsc{sr} & 1728.2 & -152.71 & 30.04 \\
154 & \textsc{sr} & \textsc{sr} & 1437.5 & -152.58 & 33.04 \\
155 & \textsc{sr} & \textsc{sr} & 1149.9 & -154.16 & 36.04 \\
156 & \textsc{sr} & \textsc{sr} & 872.7 & -159.06 & 39.04 \\
157 & \textsc{sr} & \textsc{sr} & 626.1 & -171.19 & 42.04 \\
158 & \textsc{sr} & \textsc{sr} & 470.9 & 162.06 & 45.04 \\
% 159 & \textsc{sr} & \textsc{sr} & 509.8 & 126.74 & 48.04 \\
\bottomrule
%%%%%%%%%%%%%%%%%%%%%%%%%%%55
\end{longtable}
}
\end{center}